\documentclass[11pt, reqno]{amsart}

\usepackage{amssymb}
\usepackage{hyperref}
\usepackage[matrix, arrow, curve]{xy}

\usepackage[english]{babel}
\usepackage{ulem}

\theoremstyle{plain}
\newtheorem{theorem}{Theorem}[section]

\newtheorem{corollary}[theorem]{Corollary}

\theoremstyle{remark}
\newtheorem{remark}[theorem]{Remark}

\theoremstyle{definition}
\newtheorem{example}[theorem]{Example}

\newtheorem{definition}[theorem]{Definition}

\numberwithin{equation}{section}

\begin{document}

\title[Galois groups and group actions on Lie algebras]{Galois groups and group actions on Lie algebras}

\author{A.L. Agore}
\thanks{The first named author was  partially supported by a grant of Ministery of Research and Innovation, CNCS - UEFISCDI, project number PN-III-P1-1.1-TE-2016-0124, within PNCDI III and is a FWO (Fonds voor Wetenschappelijk Onderzoek Flanders) fellow.}
\address{Vrije Universiteit Brussel, Pleinlaan 2, B-1050 Brussels, Belgium}
\address{Simion Stoilow Institute of Mathematics of the Romanian Academy, P.O. Box 1-764, 014700 Bucharest, Romania}
\email{ana.agore@gmail.com, ana.agore@vub.be}

\author{G. Militaru}
\address{Faculty of Mathematics and Computer Science, University of
Bucharest, Str. Academiei 14,
RO-010014 Bucharest 1, Romania}
\email{gigel.militaru@gmail.com}


\maketitle

\begin{abstract}
If $\mathfrak{g} \subseteq \mathfrak{h}$ is an extension of Lie
algebras over a field $k$ such that ${\rm dim}_k (\mathfrak{g}) =
n$ and ${\rm dim}_k (\mathfrak{h}) = n + m$, then the Galois group
${\rm Gal} \, (\mathfrak{h}/\mathfrak{g})$ is explicitly described
as a subgroup of the canonical semidirect product of groups ${\rm
GL} (m, \, k) \rtimes {\rm M}_{n\times m} (k)$. An Artin type
theorem for Lie algebras is proved: if a group $G$ whose order is
invertible in $k$ acts as automorphisms on a Lie algebra
$\mathfrak{h}$, then $\mathfrak{h}$ is isomorphic to a skew
crossed product $\mathfrak{h}^G \, \#^{\bullet} \, V$, where
$\mathfrak{h}^G$ is the subalgebra of invariants and $V$ is the
kernel of the Reynolds operator. The Galois group ${\rm Gal} \,
(\mathfrak{h}/\mathfrak{h}^G)$ is also computed, highlighting the
difference from the classical Galois theory of fields where the
corresponding group is $G$. The counterpart for Lie algebras of
Hilbert's Theorem 90 is proved and based on it the structure of
Lie algebras $\mathfrak{h}$ having a certain type of action of a
finite cyclic group is described. Radical extensions of finite
dimensional Lie algebras are introduced and it is shown that their
Galois group is solvable. Several applications and examples are
provided.
\end{abstract}

\section{Introduction}
The complete description of the automorphism group ${\rm Aut}_{\rm
Lie} (\mathfrak{h})$ of a given Lie algebra $\mathfrak{h}$ is an
old and notoriously difficult problem intimately related to the
structure of Lie algebras. One of the most important results
dealing with this problem shows that the automorphism group of a
finite-dimensional simple Lie algebra over an algebraically closed
field of characteristic zero is generated, with few exceptions, by
the invariant automorphisms \cite[Theorem 4]{jacobson}. This
allows for a full description of the automorphism group of any
finite dimensional reductive Lie algebra. Beyond the theoretical
interest in this problem, the description of the automorphism
group of an arbitrary Lie algebra turns out to be of crucial
importance for the construction of solutions to Einstein's field
equations for Bianchi geometries, in the study of
$(4+1)$-dimensional spacetimes with applications in cosmology
\cite{Chr, fisher, gray, petros} or for discrete symmetries of
differential equations. The classification of automorphism groups
for indecomposable real Lie algebras is known only up to dimension
six and it has been only recently finished \cite{fisher, gray}.
For other contributions to the subject see \cite{bavula, Han} and
the references therein. Perhaps the strongest motivation for
studying ${\rm Aut}_{\rm Lie} (\mathfrak{h})$ comes from Hilbert's
invariant theory, whose foundation was set at the level of Lie
algebras in the classical papers \cite{borel, bra, thrall, Wever}
- for more recent work on the subject we refer the reader to
\cite{bry, donkin, papistas}. An action as automorphisms of a
group $G$ on a Lie algebra $\mathfrak{h}$ is a morphism of groups
$\varphi: G \to {\rm Aut}_{\rm Lie} (\mathfrak{h})$. Particular
attention was given to the situation where $G$ is a finite
subgroup of ${\rm Aut}_{\rm Lie} (\mathfrak{h})$ with the
canonical action on $\mathfrak{h}$; in this case achieving the
description of the subgroups of ${\rm Aut}_{\rm Lie}
(\mathfrak{h})$ is the key step. If $\varphi: G \to {\rm Aut}_{\rm
Lie} (\mathfrak{h})$ is an action, then we can consider the
subalgebra of invariants $\mathfrak{h}^G$ and we obtain an
extension $\mathfrak{h}^G \subseteq \mathfrak{h}$ of Lie algebras.
The fundamental problem of invariant theory \cite{borel, hochster,
mont2, procesi}, in the setting of Lie algebras comes down to
finding under which conditions on $G$ and $\mathfrak{h}$ the
algebraic/geometric properties can be transferred between the two
Lie algebras $\mathfrak{h}^G$ and $\mathfrak{h}$. Turning to the
problem we started with, since describing the automorphism group
${\rm Aut}_{\rm Lie} (\mathfrak{h})$ of a given Lie algebra
$\mathfrak{h}$ is an extremely complicated task, it is natural to
start by considering only those automorphisms of $\mathfrak{h}$
which fix a given subalgebra $\mathfrak{g} \neq 0$ of
$\mathfrak{h}$. Thus, we can define the Galois group ${\rm Gal} \,
(\mathfrak{h}/\mathfrak{g})$ of the extension $\mathfrak{g}
\subseteq \mathfrak{h}$ as the subgroup of all Lie algebra
automorphisms $\sigma: \mathfrak{h} \to \mathfrak{h}$ that fix
$\mathfrak{g}$, i.e. $\sigma (g) = g$, for all $g\in
\mathfrak{g}$. In an ideal situation, after computing ${\rm Gal}
\, (\mathfrak{h}/\mathfrak{g})$ for as many subalgebras
$\mathfrak{g}$ of $\mathfrak{h}$ as possible, we will have a
complete picture on the entire group ${\rm Aut}_{\rm Lie}
(\mathfrak{h})$ as well.\\ Having defined the group ${\rm Gal} \,
(\mathfrak{h}/\mathfrak{g})$, the following question arises
naturally:

\textit{What is the counterpart in the context of Lie
algebras of the classical Galois theory for fields?}

At first sight the chances of developing a promising Galois theory
for Lie algebras are very low since even the basic concepts from
field theory such as the algebraic/separable/normal extensions,
the splitting fields of a polynomial, etc. are rather difficult to
define in the context of Lie algebras. Moreover, it is unlikely to
have a fundamental theorem establishing a bijective correspondence
between the subgroups of ${\rm Gal} \,
(\mathfrak{h}/\mathfrak{g})$ and the Lie subalgebras
$\mathfrak{g}'$ of $\mathfrak{h}$ such that $\mathfrak{g}
\subseteq \mathfrak{g}' \subseteq \mathfrak{h}$ as in the case of
the classical Galois theory (Example~\ref{thfundim}). On the other
hand, several counterparts of the classical Galois theory for
fields were proved in the context of associative algebras
\cite{Meyer}, differential Galois theory \cite{Amajid}, Hopf
algebras \cite{mont}, von Neumann algebras \cite{ocneanu},
structured ring spectra \cite{rognes} or stable homotopy theory
\cite{mat}. In this context, invariant theory seems to provide a
better approach for our problem: if $G \leq {\rm Aut} (K)$ is a
finite group of automorphisms of a field $K$ then the famous
Artin's theorem states that $k := K^G \subseteq K$ is a finite
Galois extension of degree $[K: k] = |G|$ and ${\rm Gal} (K/k) =
G$ \cite[Theorem 1.8]{Lang}. Furthermore, the  extension $k
\subseteq K$ has a normal basis as a consequence of being Galois;
that is, there exists $x\in K$ such that $\{\sigma (x) \, | \,
\sigma \in G\}$ is a $k$-basis of $K$ \cite[Theorem 1.8]{Lang}.
One of the many generalizations of Artin's theorem deals with
arbitrary actions \cite[Example 8.1.2]{mont}: if $\varphi: G \to
{\rm Aut} (K)$ is an action as automorphisms of a finite group $G$
on a field $K$, then $K/ K^G$ is a Galois extension in the
classical sense with Galois group $G$ if and only if $G$ acts
faithfully on $K$. A version of Artin's theorem for Hopf-Galois
extensions (a concept which generalizes Galois extensions for
fields \cite[Example 8.1.2]{mont}) was obtained in \cite[Theorem
1.18]{blattner}. At this level, Hopf-Galois extensions satisfying
the normal basis property coincide with crossed products
\cite[Corollary 8.2.5]{mont}. This last observation allows us to
restate Artin's theorem in a more convenient but equivalent
manner, as follows: if $ G \leq {\rm Aut} (K)$ is a finite group
of automorphisms of a field $K$, then $K$ is isomorphic to a
crossed product algebra $k \#_{\sigma} \, k[G]^*$ between the
field of invariants $k = K^G$ and the dual algebra of the group
algebra $k[G]$, associated to some cocycle $\sigma : k[G]^*
\otimes k[G]^* \to k$. With this last conclusion in mind, we
replace the category we work in: instead of fields we consider Lie
algebras together with group actions as automorphisms. Now the
question if an Artin type theorem holds for Lie algebras has a
positive answer with a slight amending though: the role of the
classical crossed product of Lie algebras (we use the terminology
of \cite[Section 4.1]{am-2013}), as it arises in the theory of
Chevalley and Eilenberg \cite{CE} will be played by a new
construction, called skew crossed product of Lie algebras, which
is introduced in Section~\ref{prel} as a generalization of the
semidirect product of Lie algebras.

The paper is organized as follows: in Section~\ref{prel} we recall
the basic concepts used throughout. The key role in our approach
is played by the recently introduced unified product of Lie
algebras (\cite{am-2013}). As a special case of the unified
product we will introduce in Example~\ref{produssucit} the skew
crossed product, denoted by $\mathfrak{g} \, \#^{\bullet} \, V $,
whose construction requires a Lie algebra $\mathfrak{g}$, a right
Lie $\mathfrak{g}$-module $(V, \, \leftharpoonup)$ equipped with a
twisted bracket operation and a cocycle $\theta :V \times V \to
\mathfrak{g}$ satisfying a set of axioms. If $\mathfrak{g}
\subseteq \mathfrak{h}$ is an extension of Lie algebras,
Theorem~\ref{grupulgal} provides the explicit description of the
Galois group ${\rm Gal} \, (\mathfrak{h}/\mathfrak{g})$ as a
subgroup of the canonical semidirect product ${\rm GL}_k (V)
\rtimes {\rm Hom}_k (V, \, \mathfrak{g})$ of groups, where $V$ is
a vector space that measures the codimension of $\mathfrak{g}$ in
$\mathfrak{h}$. i.e., the 'degree' of the extension
$\mathfrak{h}/\mathfrak{g}$. We point out that the group ${\rm
GL}_k (V) \rtimes {\rm Hom}_k (V, \, \mathfrak{g})$ is a lot more
complex than the classical general affine group ${\rm GL}_k (V)
\rtimes V$ of an affine space $V$. Theorem~\ref{recsiGal} is the
counterpart of Artin's Theorem for Lie algebras: if $G$ is a
finite group of invertible order in $k$ acting on a Lie algebra
$\mathfrak{h}$, then the Lie algebra $\mathfrak{h}$ is
reconstructed as a skew crossed product $\mathfrak{h} \cong
\mathfrak{h}^G \, \#^{\bullet} \, V$ between the Lie subalgebra of
invariants $\mathfrak{h}^G$ and the kernel $V$ of the Reynolds
operator $t: \mathfrak{h} \to \mathfrak{h}^G$. The Galois group
${\rm Gal} \, (\mathfrak{h}/\mathfrak{h}^G)$ is also described and
Example~\ref{exrelmare} shows that even in the case of faithful
actions, the group ${\rm Gal} \, (\mathfrak{h}/\mathfrak{h}^G)$ is
different from $G$, as opposed to the classical Galois theory of
fields where the two groups coincide. Theorem~\ref{Hilbert90} is
Hilbert's 90 Theorem for Lie algebras: if $G$ is a cyclic group
then the kernel of the Reynolds operator $t: \mathfrak{h} \to
\mathfrak{h}^G$ is determined. The structure of Lie algebras
$\mathfrak{h}$ endowed with a certain type of action of a finite
cyclic group is then described in Corollary~\ref{ciclicstru}:
$\mathfrak{h}$ is isomorphic to a semidirect product between
$\mathfrak{h}^G$ and an ideal of $\mathfrak{h}$. This is the Lie
algebra counterpart of the structure theorem for cyclic Galois
extensions of fields \cite[Theorem 6.2]{Lang}: if $G \leq {\rm
Aut} (K)$ is a cyclic subgroup of order $n$ of the automorphism
group of a field $K$ of characteristic zero and $k := K^G$, then
$K$ is isomorphic to the splitting field over $k$ of a polynomial
of the form $X^n - a \in k[X]$. Section~\ref{exempleconc} is
devoted to computing the Galois groups for several Lie algebra
extensions. Corollary~\ref{galcodim1} shows that if $\mathfrak{g}
\subseteq \mathfrak{h}$ is a Lie subalgebra of codimension $1$ in
$\mathfrak{h}$, then the Galois group ${\rm Gal} \,
(\mathfrak{h}/\mathfrak{g})$ is metabelian (in particular,
solvable). Based on this, the Lie algebra counterpart of the
concept of a radical extension of fields is proposed in
Definition~\ref{radical}. As in the classical Galois theory,
Theorem~\ref{galcodim1aa} proves that the Galois group ${\rm Gal}
\, (\mathfrak{h}/\mathfrak{g})$ of a radical extension
$\mathfrak{g} \subseteq \mathfrak{h}$ of finite dimensional Lie
algebras is a solvable group. Several other applications and
concrete examples of Galois groups are presented. For instance, in
Example~\ref{galtrivial} we present a Lie algebra extension
$\mathfrak{g} \subseteq \mathfrak{h}$ whose Galois group is
trivial i.e., ${\rm Gal} \, (\mathfrak{h}/\mathfrak{g}) = \{\rm
Id_{\mathfrak{h}}\}$. However, we point out that in general the
Galois group of a Lie algebra extension is far from being trivial.

\section{Preliminaries}\label{prel}
\subsection
{Notations and terminology}
All vector spaces, (bi)linear maps or Lie algebras are over an
arbitrary field $k$. A map $f: V \to W$ between two vector spaces
is called trivial if $f (v) = 0$, for all $v\in V$. For two vector
spaces $V$ and $W$ we denote by ${\rm Hom}_k (V, \, W)$ the
abelian group of all linear maps from $V$ to $W$ and by ${\rm
GL}_k (V) := {\rm Aut}_k (V)$ the group of all linear
automorphisms of $V$; if $V$ has dimension $m$ over $k$ then ${\rm
GL}_k (V)$ is identified with the general linear group ${\rm GL}
(m, \, k)$ of all $m\times m$ invertible matrices over $k$. As
usual, ${\rm SL} (m, \, k)$ stands for the special linear group of
degree $m$ over $k$ which is the normal subgroup of ${\rm GL} (m,
\, k)$ consisting of all $m \times m$ matrices of determinant $1$.
Throughout this paper we use the right hand side convention for
the semidirect products whose construction we briefly recall
below. Let $G$ and $H$ be two groups and $\triangleleft : H \times
G \to H$ a right action as automorphisms of the group $G$ on the
group $H$, i.e. the following compatibility conditions hold for
all $g$, $g' \in G$ and $h$, $h'\in H$:
\begin{equation*}\label{actiune}
h \triangleleft 1 = h, \qquad h \triangleleft (gg') =
(h\triangleleft g) \triangleleft g' \qquad (hh') \triangleleft g =
(h\triangleleft g) (h'\triangleleft g)
\end{equation*}
The associated semidirect product $G \rtimes H$ is the group
structure on $G \times H$ with multiplication given for any $g$,
$g' \in G$ and $h$, $h'\in H$ by:
\begin{equation}\label{prodsemidr}
(g, \, h) \cdot (g', \, h') := \bigl(gg', \, (h\triangleleft g')
h' \bigl)
\end{equation}
Let $V$ and $W$ be two vector spaces. Then there exists a
canonical right action as automorphisms of the group ${\rm GL}_k
(V)$ on the abelian group $\bigl({\rm Hom}_k (V, \, W), \, +
\bigl)$ given for any $r \in {\rm Hom}_k (V, \, W)$ and $\sigma
\in {\rm GL}_k (V)$ by:
$$
\triangleleft: \, {\rm Hom}_k (V, \, W) \times {\rm GL}_k (V) \to
{\rm Hom}_k (V, \, W), \qquad r \triangleleft \sigma := r\circ
\sigma
$$
We shall denote by ${\mathbb G}_W^V := {\rm GL}_k (V) \rtimes {\rm
Hom}_k (V, \, W)$ the corresponding semidirect product, i.e.
${\mathbb G}_W^V  := {\rm GL}_k (V) \times {\rm Hom}_k (V, \, W)
$, with the multiplication given for any $\sigma$, $\sigma' \in
{\rm GL}_k (V)$ and $r$, $r'\in {\rm Hom}_k (V, \, W)$ by:
\begin{equation}\label{grupstrc}
(\sigma, \, r ) \cdot (\sigma', \, r') := (\sigma \circ \sigma',
\, r \circ \sigma' + r')
\end{equation}
The unit of the group ${\mathbb G}_W^V$ is $({\rm Id}_V, \, 0)$.
Moreover, ${\rm GL}_k (V) \cong {\rm GL}_k (V)\times \{0\}$ is a
subgroup of ${\mathbb G}_W^V$ and the abelian group ${\rm Hom}_k
(V, \, W) \cong \{{\rm Id}_V\} \times {\rm Hom}_k (V, \, W)$ is a
normal subgroup of ${\mathbb G}_W^V$. The relation $(\sigma, \, r)
= (\sigma, \, 0) \cdot ({\rm Id}_V, \, r)$ gives an exact
factorization ${\mathbb G}_W^V = {\rm GL}_k (V) \, \cdot \, {\rm
Hom}_k (V, \, W) $ of the group ${\mathbb G}_W^V$ through the
subgroup ${\rm GL}_k (V)$ and the abelian normal subgroup ${\rm
Hom}_k (V, \, W)$. Being a semidirect product, the group ${\mathbb
G}_W^V$ is a split extension of ${\rm GL}_k (V)$ by the abelian
group ${\rm Hom}_k (V, \, W)$; that is, it fits into an exact
sequence of groups $ 0 \to {\rm Hom}_k (V, \, W) \to {\mathbb
G}_W^V \to {\rm GL}_k (V) \to 1$ and the canonical projection
${\mathbb G}_W^V \to {\rm GL}_k (V)\to 1$ has a section that is a
morphism of groups. The group ${\mathbb G}_W^V$ constructed above
will play a crucial role in the paper since the Galois group of an
arbitrary extension of Lie algebras embeds in such a group. If $V
\cong k$ is a $1$-dimensional vector space, then the group
${\mathbb G}_W^k$ identifies with  the semidirect product $k^*
\rtimes W$ of the multiplicative group of units $(k^*, \cdot)$
with the abelian group $(W, +)$ and will be denoted simply by
${\mathbb G}_W$. The multiplication on ${\mathbb G}_W = k^*
\rtimes W$ is given for any $u$, $u'\in k^*$ and $x$, $x' \in W$
by:
\begin{equation}\label{grupstrcb}
(u, \, x ) \cdot (u', \, x') := (uu', \, u'x + x')
\end{equation}
The non-abelian group ${\mathbb G}_W$ is an extension of the
abelian group $k^*$ by the abelian group $W = (W, +)$; hence,
${\mathbb G}_W$ is a metabelian group (i.e., the derived
subgroup $[{\mathbb G}_W, \, {\mathbb G}_W ]$ is abelian). In
particular, ${\mathbb G}_W$ is a $2$-step solvable group. On the
other hand, if $W \cong k$ is a $1$-dimensional vector space then
${\mathbb G}_k^V  = {\rm GL}_k (V) \rtimes V^*$, and for finite
dimensional vector spaces $V$ the group can be identified with the
general affine group ${\rm Aff} \, (V) = {\rm GL}_k (V) \rtimes
V$.

\subsection{Groups acting on Lie algebras}
For all basic notions and undefined concepts pertaining to Lie algebra theory we
refer the reader to \cite{H, jacobson}. We denote by
$\mathfrak{gl} (m, k)$ (resp. $\mathfrak{sl} (m, k)$) the general
(resp. special) linear Lie algebra of all $m\times m$ matrices
(resp. all $m\times m$ matrices of trace $0$) having the bracket
$[A, \, B] := AB - BA$. Representations of a Lie algebra
$\mathfrak{g}$ will be viewed as right Lie $\mathfrak{g}$-modules:
a right Lie $\mathfrak{g}$-module is a vector space $V$ together
with a bilinear map $ \leftharpoonup \, : V \times \mathfrak{g}
\to V$ such that $x \leftharpoonup [a, \, b] = (x \leftharpoonup
a) \leftharpoonup b \,  - \, (x \leftharpoonup b) \leftharpoonup
a$, for all $a$, $b \in \mathfrak{g}$ and $x\in V$. Let $G$ be a
group, $\mathfrak{h}$ a Lie algebra and ${\rm Aut}_{\rm Lie}
(\mathfrak{h})$ the group of all Lie algebra automorphisms of
$\mathfrak{h}$. If $\varphi: G \to {\rm Aut}_{\rm Lie}
(\mathfrak{h})$ is a morphism of groups we will say that $G$
acts as automorphisms on $\mathfrak{h}$ and we shall denote
$\varphi (g) (x) = g \triangleright x$, for all $g\in G$ and $x\in
\mathfrak{h}$. The action is called faithful if $\varphi$ is
injective. Since $\varphi(g)$ is a Lie algebra map we have that $g
\triangleright [x, \, y] = [g\triangleright x, \, g \triangleright
y]$, for all $g\in G$ and $x$, $y\in \mathfrak{h}$. The subalgebra
of invariants $\mathfrak{h}^G$ of the action $\varphi$ of $G$ on
$\mathfrak{h}$ is defined by:
$$
\mathfrak{h}^G := \{ x \in \mathfrak{h} \, | \, g \triangleright x
= x, \forall \, g\in G \}
$$
Then $\mathfrak{h}^G \subseteq \mathfrak{h}$ is a Lie subalgebra
of $\mathfrak{h}$. If $G$ is a finite group and $|G|$ is
invertible in the base field $k$ then the trace map or
Reynolds operator (we borrowed the terminology from the
classical invariant theory of groups acting on associative
algebras \cite{hochster}) defined for any $x\in \mathfrak{h}$ by:
\begin{equation}\label{trace}
t = t_{\triangleright}: \, \mathfrak{h} \to \mathfrak{h}^G, \qquad
t (x) := |G|^{-1} \, \sum_{g\in G} \, g\triangleright x
\end{equation}
is a linear retraction of the canonical inclusion $\mathfrak{h}^G
\hookrightarrow \mathfrak{h}$. Furthermore, for any $a \in
\mathfrak{h}^G$ and $x \in \mathfrak{h}$ we have $t ([a, \, x]) =
[a, \, t(x)]$.

\begin{example}\label{exgract}
(1) The basic example of a group
acting on a Lie algebra is
provided by any subgroup $G$ of ${\rm Aut}_{\rm Lie}
(\mathfrak{h})$ with the canonical action given by $\sigma
\triangleright x : = \sigma (x)$, for all $\sigma \in G$ and $x\in
\mathfrak{h}$. Automorphic Lie algebras \cite{lomb} introduced in
the context of integrable systems are examples of Lie algebras of
invariants - for further details we refer to \cite{kni}.

(2) The group ${\rm GL} (n, \, k)$ acts on $\mathfrak{gl} (n, k)$
(resp. $\mathfrak{sl} (n, k)$) by conjugation, i.e. $U
\triangleright X : = U X U^{-1}$, for all $U \in {\rm GL} (n, \,
k)$ and $X \in \mathfrak{gl} (n, k)$ (resp. $X \in \mathfrak{sl}
(n, k)$). Thus, any subgroup of ${\rm GL} (n, \, k)$ (such as
${\rm SL} (n, \, k)$, the permutation group $S_n$ on $n$ letters,
the cyclic group $C_{n}$ or more generally any finite group of
order $n$) acts on the Lie algebras $\mathfrak{gl} (n, k)$ and
$\mathfrak{sl} (n, k)$ by the same action. The subalgebras of
invariants for these actions are exactly the centralizers  in
$\mathfrak{gl} (n, k)$ (resp. $\mathfrak{sl} (n, k)$) of ${\rm GL}
(n, \, k)$ (or its subgroups). For example, the symmetric group
$S_n$ acts as automorphisms on $\mathfrak{gl} (n, k)$ via the
action:
\begin{equation}\label{symact}
S_n \to {\rm Aut}_{\rm Lie} (\mathfrak{gl} (n, k)), \quad \tau
\triangleright e_{ij} : = (e_{1 \tau (1)} + \cdots + e_{n \tau
(n)}) \, e_{ij} \, (e_{1 \tau (1)} + \cdots + e_{n \tau (n)})^{-1}
\end{equation}
for all $\tau \in S_n$ and $i$, $j = 1, \cdots, n$, where $e_{ij}$
is $n\times n$ matrix which has $1$ in the $(i, j)^{th}$-position
and zeros elsewhere. We will describe the subalgebras of
invariants $\mathfrak{gl} (n, k)^{C_n}$ and respectively
$\mathfrak{gl} (n, k)^{S_n}$ of the action defined by \ref{symact}. We start
by looking at $\mathfrak{gl} (n, k)^{C_n}$, where the
cyclic group $C_{n}$ is considered to be the subgroup of $S_{n}$ generated by
the cycle $(12...n)$. An easy computation gives:
\begin{equation}\label{inv002}
(1\, 2\, ... \,n)  \triangleright A =  (e_{12} + e_{23} + e_{34} +
\cdots + e_{n1}) A (e_{21} + e_{32} + e_{43} + \cdots + e_{1n})
\end{equation}
It follows that $A = \sum_{i, j = 1}^n \, a_{ij} \, e_{ij} \in
\mathfrak{gl} (n, k)^{C_n}$ if and only if $(12...n)
\triangleright A = A$. This yields:
\begin{equation}\label{cc100}
a_{11} = a_{nn}, \quad a_{1j} = a_{n, j-1}, \quad a_{j 1} =
a_{j-1, n}, \quad a_{i j} = a_{i-1, j-1}, \, {\rm for} \,
{\rm all} \,\, i, \,j\geqslant 2
\end{equation}
Therefore, by a careful analysis of the above compatibilities, we
obtain that $\mathfrak{gl} (n, k)^{C_n}$ is the $n$-dimensional
subalgebra of $\mathfrak{gl} (n, k)$ consisting of all $n\times
n$-matrices of the form:
$$
\begin{pmatrix}
a_1 & a_2 & \cdots & a_n\\
a_n & a_1 & \cdots & a_{n-1} \\
a_{n-1} & a_n & \cdots & a_{n-2}\\
\cdot & \cdot & \cdots & \cdot \\
a_2 & a_3 & \cdots & a_1
\end{pmatrix}
$$
for all $a_1, \cdots, a_n \in k$. Next in line is $\mathfrak{gl}
(n, k)^{S_n}$. As $S_n$ is generated by the transposition $(12)$
and the cycle $(12...n)$ it follows that a matrix $A = \sum_{i, j
= 1}^n \, a_{ij} \, e_{ij} \in \mathfrak{gl} (n, k)^{S_n}$ if and
only if $\tau \triangleright A = A$, for $\tau = (1\, 2)$ and
$\tau = (1\, 2\, ... \,n)$. Using again the formula given in (\ref{symact}) we obtain:
\begin{equation}\label{inv001}
(1\, 2) \triangleright A = (e_{12} + e_{21} + e_{33} + \cdots +
e_{nn}) A (e_{12} + e_{21} + e_{33} + \cdots + e_{nn})
\end{equation}
which yields:
\begin{eqnarray*}
a_{11} = a_{22}, \quad a_{12} = a_{21}, \quad a_{1j} = a_{2j},
\quad a_{j1} = a_{j2}, \, {\rm for} \, {\rm all} \,\,
j\geqslant 3.
\end{eqnarray*}
The above compatibilities together with those in
equation~(\ref{cc100}) come down to: $a_{ii} = \alpha \in k$ and
$a_{i j} = \beta \in k$, for all $i$, $j = 1, \cdots, n, \, i \neq
j$. Thus, $\mathfrak{gl} (n, k)^{S_n} = \big\{ \alpha I_n + \beta
\, \sum_{\substack{i, \, j=1\\ i \neq j }}^{n}  e_{ij} ~|~ \alpha,
\, \beta \in k \big\}$. Both subalgebras of invariants
$\mathfrak{gl} (n, k)^{C_n}$ and respectively $\mathfrak{gl} (n,
k)^{S_n}$ are abelian.

(3) The actions as automorphism of abelian groups on Lie algebras
can be seen as the dual concept of the well studied gradings on
Lie algebras: for an overview and the importance of the problem
introduced by Kac \cite{Kac} we refer to \cite{Eld, Koch, svob}
and the references therein.

Let $G = (G, +)$ be an abelian group and $\hat{G}$ be the group of
characters on $G$, i.e. all morphisms of groups $G \to k^*$. A $G$-graded Lie algebra is a Lie algebra
$\mathfrak{h}$ such that $\mathfrak{h} = \oplus_{g\in G} \,
\mathfrak{h}_g$, where any $\mathfrak{h}_g$ is a subspace of
$\mathfrak{h}$ such that $[\mathfrak{h}_g, \, \mathfrak{h}_{g'} ]
\subseteq \mathfrak{h}_{g + g'}$, for all $g$, $g' \in G$. If
$\mathfrak{h} = \oplus_{g\in G} \, \mathfrak{h}_g$ is a $G$-graded
Lie algebra then the map
$$
\varphi: \hat{G} \to {\rm Aut}_{\rm Lie} (\mathfrak{h}), \qquad
\varphi(\chi) (x_g) := \chi (g) \, x_g
$$
for all $\chi \in \hat{G}$, $g\in G$ and $x_g \in \mathfrak{h}_g$
is a faithful action of $\hat{G}$ on $\mathfrak{h}$. Conversely,
if $\varphi: \hat{G} \to {\rm Aut}_{\rm Lie} (\mathfrak{h})$ is an
injective morphism of groups, then $\mathfrak{h} = \oplus_{g\in G}
\, \mathfrak{h}_g$ is a $G$-graded Lie algebra where
$\mathfrak{h}_g := \{ y \in \mathfrak{h} \, | \, \chi
\triangleright y = \chi(g) y, \, \, \forall \chi \in \hat{G} \}$.
In some special cases we can say more. For instance, if $k$ is an
algebraically closed field of characteristic zero and $G$ is a
finitely generated abelian group, then there exists a one-to-one
correspondence between the set of all $G$-gradings on a given Lie
algebra $\mathfrak{h}$ and the set of all faithful actions
$\hat{G} \to {\rm Aut}_{\rm Lie} (\mathfrak{h})$ of $\hat{G}$ on
$\mathfrak{h}$ \cite[Proposition 4.1]{Koch}. Working with actions
instead of gradings comes with the advantage of not assuming the
group $G$ to be abelian nor the actions to be faithful.

(4) As a special case of (3) let us take $\mathfrak{h} =
\oplus_{i\in \mathbb Z} \, \mathfrak{h}_i$ to be a $\mathbb
Z$-graded Lie algebra. Then the multiplicative group of units
$k^*$ acts on $\mathfrak{h}$ via the following morphism of groups:
\begin{equation}\label{actgrad}
\varphi: k^* \to {\rm Aut}_{\rm Lie} (\mathfrak{h}), \qquad
\varphi(u) (y_i) := u^i \, y_i
\end{equation}
for all $u\in k^*$, $i\in \mathbb Z$ and $y_i \in \mathfrak{h}_i$
a homogeneous element of degree $i$. Moreover, the subalgebra of
invariants $\mathfrak{h}^{k^*} = \mathfrak{h}_0$, the Lie
subalgebra of all elements of degree zero. The typical example of
a $\mathbb Z$-graded Lie algebra is the Witt algebra
$\mathfrak{W}$ which is the vector space having $\{e_i \, | \, i
\in \mathbb Z\}$ as a basis and the bracket $[e_i, \, e_j] :=
(i-j) \, e_{i+j}$, for all $i$, $j\in \mathbb Z$. Another example
is given by $\mathfrak{h} := \mathfrak{sl} (2, k)$, the Lie
algebra with basis $\{e_1, \, e_2, \, e_3\}$ and the usual bracket
$[e_1, \, e_2] = e_3$, $[e_1, \, e_3] = -2 \, e_1$ and $[e_2, \,
e_3] = 2 \, e_2$ viewed with the standard grading: namely $e_1$
has degree $-1$, $e_2$ has degree $1$ and $e_3$ has degree $0$. We
obtain that the group $k^*$ acts on $\mathfrak{sl} (2, k)$ via:
\begin{equation}\label{sl2}
\varphi: k^* \to {\rm Aut}_{\rm Lie} \bigl( \mathfrak{sl} (2, k)
\bigl), \qquad u \triangleright (\alpha e_1 + \beta e_2 + \gamma
e_3) := u^{-1} \alpha e_1 + u \beta e_2 + \gamma e_3
\end{equation}
for all $u\in k^*$ and $\alpha$, $\beta$, $\gamma \in k$. The
algebra of invariants $\mathfrak{sl} (2, k)^{k^*}$ is the abelian
Lie algebra having $e_3$ as a basis.
\end{example}

\subsection{Unified products and skew crossed product for Lie algebras}
We recall from \cite{am-2013} some concepts that will play a key
role in the paper. Let $\mathfrak{g} = (\mathfrak{g}, \, [-, \,
-])$ be a Lie algebra and $V$ a vector space. A Lie extending system of $\mathfrak{g}$ through $V$ is a
system $\Lambda (\mathfrak{g}, \, V) = \bigl(\leftharpoonup, \,
\rightharpoonup, \, \theta, \{-, \, -\} \bigl)$ consisting of four
bilinear maps $\leftharpoonup: V \times \mathfrak{g} \to V$, \,
$\rightharpoonup : V \times \mathfrak{g} \to \mathfrak{g}$, \,
$\theta: V\times V \to \mathfrak{g}$, \, $\{-, \, -\} : V\times V
\to V$ satisfying the following compatibility conditions for any
$a$, $b \in \mathfrak{g}$, $x$, $y$, $z \in V$:
\begin{enumerate}
\item[(L1)] $(V, \, \leftharpoonup)$ is a right Lie
$\mathfrak{g}$-module, $ \theta (x, \, x) = 0$ and $\{x, \, x \} =
0$

\item[(L2)] $x \rightharpoonup [a, \, b] = [x \rightharpoonup a,
\, b] + [a, \, x \rightharpoonup b] + (x \leftharpoonup a)
\rightharpoonup b - (x \leftharpoonup b) \rightharpoonup a$

\item[(L3)] $\{x, \, y \} \leftharpoonup a = \{x, \, y
\leftharpoonup a \} + \{x \leftharpoonup a, \, y \} + x
\leftharpoonup (y \rightharpoonup a) - y \leftharpoonup (x
\rightharpoonup a)$

\item[(L4)] $\{x,\, y \} \rightharpoonup a = x \rightharpoonup (y
\rightharpoonup a) - y \rightharpoonup (x \rightharpoonup a) + [a,
\, \theta (x,\, y)] + \theta (x, y \leftharpoonup a) +\\ + \theta
(x \leftharpoonup a, y)$

\item[(L5)] $\sum_{(c)} \theta \bigl(x, \, \{y, \, z \}\bigl) \, +
\,  \sum_{(c)} x \rightharpoonup \theta (y, z) = 0$

\item[(L6)] $\sum_{(c)} \{x, \, \{y, \, z\}\} \, + \,  \sum_{(c)}
x \leftharpoonup \theta (y, z) = 0$
\end{enumerate}
where $\sum_{(c)}$ denotes the circular sum. The bilinear maps
$\leftharpoonup$ and $\rightharpoonup$ are called the
actions of $\Lambda (\mathfrak{g}, \, V)$ and $\theta$ is
called the cocycle of $\Lambda (\mathfrak{g}, \, V)$. Let
$\Lambda (\mathfrak{g}, \, V) = \bigl(\leftharpoonup,
\rightharpoonup,\, \theta, \{-, \, -\} \bigl)$ be an extending
system of $\mathfrak{g}$ through $V$ and let $ \mathfrak{g}
\,\natural \, V = \mathfrak{g} \,\natural_{\Lambda (\mathfrak{g},
V)} \, V $ be the vector space $\mathfrak{g} \, \times V$ with the
bracket $[ -, \, -]$ defined for any $a$, $b \in \mathfrak{g}$ and
$x$, $y \in V$ by:
\begin{equation}\label{brackunif}
[(a, x), \, (b, y)] := \bigl( [a, \, b] + x \rightharpoonup b - y
\rightharpoonup a + \theta (x, y), \,\, \{x, \, y \} +
x\leftharpoonup b - y \leftharpoonup a \bigl)
\end{equation}
Then $\mathfrak{g} \,\natural \, V $ is a Lie algebra
\cite[Theorem 2.2]{am-2013} called the unified product of
$\mathfrak{g}$ and $\Lambda (\mathfrak{g}, \, V)$, and contains
$\mathfrak{g} \cong \mathfrak{g} \times \{0\}$ as a Lie
subalgebra. Conversely, let $\mathfrak{g}$ be a Lie algebra, $E$ a
vector space such that $\mathfrak{g}$ is a subspace of $E$. Then,
any Lie algebra structure $[-, \, -]_E$ on $E$ containing
$\mathfrak{g}$ as a Lie subalgebra is isomorphic to a unified
product: i.e., $(E, [-, \, -]_E) \cong \mathfrak{g} \, \natural \,
V$, for some extending system $\Lambda (\mathfrak{g}, \, V) =
\bigl(\leftharpoonup, \, \rightharpoonup, \, \theta, \{-, \, -\}
\bigl)$ of $\mathfrak{g}$ through $V$ (\cite[Theorem
2.4]{am-2013}). As explained in \cite{am-2013}, the well known
bicrossed product as well as the crossed product of Lie algebras
are special cases of unified products. First of all, we observe
that the extending system of a Lie algebra $\mathfrak{g}$ through
$V$ is a cocycle deformation of the concept of matched pair
between two Lie algebras, as introduced in \cite{LW, majid}.
Indeed, if $\theta$ is the trivial map, then $\Lambda
(\mathfrak{g}, \, V) = \bigl(\leftharpoonup, \, \rightharpoonup,
\, \theta := 0, \, \{-, \, -\} \bigl)$ is a Lie extending system
of $\mathfrak{g}$ through $V$ if and only if $(V, \, \{-, \, -\})$
is a Lie algebra and $(\mathfrak{g}, V, \leftharpoonup, \,
\rightharpoonup)$ is a matched pair of Lie algebras. In this case,
the associated unified product $\mathfrak{g} \,\natural \, V =
\mathfrak{g} \bowtie V $ is precisely the bicrossed product
(also called bicrossproduct in \cite[Theorem 4.1]{majid}
and double Lie algebra in \cite[Definition 3.3]{LW})
associated to the matched pair $(\mathfrak{g}, V, \leftharpoonup,
\, \rightharpoonup)$. Secondly,  if $\leftharpoonup$ is the
trivial map, then $\Lambda (\mathfrak{g}, \, V) =
\bigl(\leftharpoonup := 0, \, \rightharpoonup, \, \theta, \, \{-,
\, -\} \bigl)$ is a Lie extending system of $\mathfrak{g}$ through
$V$ if and only if $(V, \{-, -\})$ is a Lie algebra and the
following four compatibilities hold for any $g$, $h\in
\mathfrak{g}$ and $x$, $y$, $z\in V$:
\begin{eqnarray*}
&&  f(x,\, x) = 0\\
&& \{x, \, y \} \rightharpoonup g = x
\rightharpoonup (y \rightharpoonup g) - y \rightharpoonup (x
\rightharpoonup g) + [g, \, \theta (x, \, y)]\\
&&  x \rightharpoonup [g, \, h] = [x \rightharpoonup g, \, h] +
[g, \, x \rightharpoonup h]\\
&& \sum_{(c)} \, \theta (x,
\{y, z \}) + \sum_{(c)} \, x\rightharpoonup \theta (y, z) =0
\end{eqnarray*}
In this case, the associated unified product $\mathfrak{g}
\,\natural \, V = \mathfrak{g} \# \, V $ is the classical
crossed product of the Lie algebras $\mathfrak{g}$ and $V$
introduced in \cite{CE} in connection to the extension problem.
The Lie algebra $\mathfrak{g} \# \, V $ is an extension of $V$ by
$\mathfrak{g}$, which in an ideal of it. Although it is completely
different from both the crossed and the bicrossed product, the
following construction is also a special case of the unified
product:

\begin{example}\label{produssucit}
Consider the bilinear map $\rightharpoonup : V \times \mathfrak{g}
\to \mathfrak{g}$ to be trivial, i.e. $x \rightharpoonup g = 0$,
for all $x\in V$ and $g \in \mathfrak{g}$. Then $\Lambda
(\mathfrak{g}, \, V) = \bigl(\leftharpoonup, \, \rightharpoonup :=
0, \, \theta, \{-, \, -\} \bigl)$ is a Lie extending system of
$\mathfrak{g}$ through $V$ if and only if the following
compatibility conditions hold for any $a \in \mathfrak{g}$, $x$,
$y$, $z \in V$:
\begin{enumerate}
\item[(T1)] \, $(V, \, \leftharpoonup)$ is a right Lie
$\mathfrak{g}$-module, $ \theta (x, \, x) = 0$ and $\{x, \, x \} =
0$

\item[(T2)] \, $\{x, \, y \} \leftharpoonup a = \{x, \, y
\leftharpoonup a \} + \{x \leftharpoonup a, \, y \}$

\item[(T3)] \,  $[\theta (x, \, y), \, a]  =  \theta (x, \, y
\leftharpoonup a) + \theta (x \leftharpoonup a, \, y)$

\item[(T4)] \, $\sum_{(c)} \theta \bigl(x, \, \{y, \, z \}\bigl) =
0$

\item[(T5)] \,  $\sum_{(c)} \{x, \, \{y, \, z\}\} \, + \,
\sum_{(c)} x \leftharpoonup \theta (y, \, z) = 0$
\end{enumerate}
In this case the trivial map $\rightharpoonup$ will be omitted
when writing down the Lie extending system $\Lambda (\mathfrak{g},
\, V)$. The associated unified product $\mathfrak{g} \, \natural
\, V $ will be denoted by $\mathfrak{g} \, \#^{\bullet} \, V $ and
we will call it the skew crossed product associated to the
system $\Lambda (\mathfrak{g}, \, V) = \bigl(\leftharpoonup, \,
\theta, \{-, \, -\} \bigl)$ satisfying (T1)$-$(T5). Thus,
$\mathfrak{g} \, \#^{\bullet} \, V $ is the vector space
$\mathfrak{g} \, \times V$ with the Lie bracket $[ -, \, -]$
defined for any $a$, $b \in \mathfrak{g}$ and $x$, $y \in V$ by:
\begin{equation}\label{brackskcp}
[(a, x), \, (b, y)] := \bigl( [a, \, b] + \theta (x, y), \,\, \{x,
\, y \} + x\leftharpoonup b - y \leftharpoonup a \bigl)
\end{equation}
As already mentioned, the skew crossed product $\mathfrak{g} \,
\#^{\bullet} \, V $ is completely different from both the crossed
as well as the bicrossed product of Lie algebras: in the
construction of $\mathfrak{g} \, \#^{\bullet} \, V $ the bilinear
map $\{-, \, -\}$ on $V$ is not a Lie bracket (axiom (T5) is a
deformation of the Jacobi identity) and moreover $\mathfrak{g}
\cong \mathfrak{g} \times \{0\}$ is only a subalgebra in
$\mathfrak{g} \, \#^{\bullet} \, V $, not an ideal. An explicit
example of a skew crossed product is given in Example~\ref{priexe}
where we write $\mathfrak{sl} (2, k)$ as a skew crossed product $k
\#^{\bullet}\, k^2$  between the abelian Lie algebras of
dimensions one and two, associated to a certain right action
$\leftharpoonup$ and a cocylce $\theta$.

Moreover, if the cocycle $\theta$ of a Lie extending structure
$\Lambda (\mathfrak{g}, \, V) = \bigl(\leftharpoonup, \, \theta,
\{-, \, -\} \bigl)$ is also the trivial map, then the skew crossed
product $\mathfrak{g} \, \#^{\bullet} \, V $ is just the usual
semidirect product $\mathfrak{g} \, \rtimes \, V $ of two Lie
algebras written in the right side convention. We point out that
in our notational convention the Lie algebra $\mathfrak{g} \,
\rtimes \, V $ contains $V \cong \{0 \} \times V$ as an ideal.
\end{example}

\section{The Galois group of Lie algebra extensions}\label{GalLie}
Let $\mathfrak{g} \subseteq \mathfrak{h}$ be an extension of Lie
algebras. We define the Galois group ${\rm Gal} \,
(\mathfrak{h}/\mathfrak{g})$ as the subgroup of ${\rm Aut}_{\rm
Lie} (\mathfrak{h})$ consisting of all Lie algebra automorphisms
of $\mathfrak{h}$ that fix $\mathfrak{g}$, i.e.:
$$
{\rm Gal} \, (\mathfrak{h}/\mathfrak{g}) := \{ \sigma \in {\rm
Aut}_{\rm Lie} (\mathfrak{h}) \, | \, \sigma (g) = g, \, \forall
\, g\in \mathfrak{g} \}
$$
Since ${\rm Gal} \, (\mathfrak{h}/\mathfrak{g}) \leq {\rm
Aut}_{\rm Lie} (\mathfrak{h})$ we can consider the subalgebra of
invariants $\mathfrak{h}^{{\rm Gal} (\mathfrak{h}/\mathfrak{g})}$.
Of course, we have that $\mathfrak{g} \subseteq \mathfrak{h}^{
{\rm Gal} (\mathfrak{h}/\mathfrak{g})}$. As it can be seen from
the example below, a fundamental theorem establishing a bijective
correspondence between the subgroups of ${\rm Gal} \,
(\mathfrak{h}/\mathfrak{g})$ and the Lie subalgebras $\mathfrak{g}'$ of
$\mathfrak{h}$ such that $\mathfrak{g} \subseteq \mathfrak{g}'
\subseteq \mathfrak{h}$ does not hold in the context of Lie
algebras.

\begin{example}\label{thfundim}
Let $\mathfrak{h} := \mathfrak{aff} (2, k)$ be the $2$-dimensional
affine Lie algebra with basis $\{e_1, \, e_2 \}$ and bracket
$[e_1, \, e_2] = e_2$ and $\mathfrak{g} := ke_1$ the abelian Lie
subalgebra. Then ${\rm Gal} \, (\mathfrak{aff}(2,
k)/\mathfrak{g})$ is isomorphic to $k^*$ the multiplicative group of units of
$k$ while the subalgebra of invariants $\mathfrak{aff} (2, k)^{k^*}
= \mathfrak{g}$. Of course, between $\mathfrak{g}$ and
$\mathfrak{aff} (2, k)$ there are no proper intermediary
subalgebras while $k^*$ has many subgroups (such as the
cyclic groups $U_n (k)$ of $n$-roots of unity) whose subalgebras
of invariants coincide with $\mathfrak{g}$.
\end{example}

In what follows we will describe the group ${\rm Gal} \,
(\mathfrak{h}/\mathfrak{g})$. First we fix a linear map $p:
\mathfrak{h} \to \mathfrak{g}$ such that $p(g) = g$, for all $g
\in \mathfrak{g}$ - such a map always exists as $k$ is a field.
Then $V := {\rm Ker} (p)$ is a subspace of $\mathfrak{h}$ and a
complement of $\mathfrak{g}$ in $\mathfrak{h}$, that is
$\mathfrak{h} = \mathfrak{g} + V$ and $\mathfrak{g} \cap V =
\{0\}$. Using $p$ we define a Lie extending system of
$\mathfrak{g}$ through $V$, called the canonical extending
system associated to $p$, where the bilinear maps
$\rightharpoonup \, : V \times \mathfrak{g} \to \mathfrak{g}$,
$\leftharpoonup \, : V \times \mathfrak{g} \to V$, $\theta : V
\times V \to \mathfrak{g}$ and  $\{\, , \, \} : V \times V \to V$
are given by the following formulas \cite[Theorem 2.4]{am-2013}
for any $g \in \mathfrak{g}$ and $x$, $y\in V$:
\begin{eqnarray}
x \rightharpoonup g &:=& p \bigl([x, \,g]\bigl), \qquad
x \leftharpoonup g := [x, \, g] - p \bigl([x, \, g]\bigl) \label{can1}\\
\theta(x, y) &:=& p \bigl([x, \, y]\bigl), \qquad  \{x, y\} := [x,
\, y] - p \bigl([x, \, y]\bigl) \label{can2}
\end{eqnarray}
Thus we can construct the unified product $\mathfrak{g} \,\natural
\, V$ associated to the canonical extending structure, which is a
Lie algebra with the bracket given by the formula (\ref{brackunif}). The map
$\varphi: \mathfrak{g} \,\natural \, V \to \mathfrak{h}$, defined by
$\varphi(g, x) := g + x$, is an isomorphism of Lie algebras with
the inverse given by $\varphi^{-1}(y) := \bigl(p(y), \, y -
p(y)\bigl)$, for all $y \in \mathfrak{h}$. Since $\varphi$ fixes
$\mathfrak{g} \cong \mathfrak{g} \times \{0\}$ we obtain that the
map
\begin{equation}\label{izogalgen}
{\rm Gal} \, (\mathfrak{h}/\mathfrak{g}) \to {\rm Gal} \, (
\mathfrak{g} \,\natural \, V /\mathfrak{g}), \qquad \sigma \mapsto
\varphi^{-1} \circ \sigma \circ \varphi
\end{equation}
is an isomorphism of groups with the inverse given by $\psi
\mapsto \varphi \circ \psi \circ \varphi^{-1}$. It follows from
\cite[Lemma 2.5]{am-2013} that there exists a bijection between
the set of all elements $\psi \in {\rm Gal} \, ( \mathfrak{g}
\,\natural \, V /\mathfrak{g})$ and the set of all pairs $(\sigma,
r) \in {\rm GL}_k (V) \times {\rm Hom}_k (V, \, \mathfrak{g})$,
satisfying the following four compatibility conditions for any $g
\in \mathfrak{g}$, $x$, $y \in V$:
\begin{enumerate}
\item[(G1)] $ \sigma (x \leftharpoonup g) =  \sigma (x)
\leftharpoonup g$, that is $\sigma : V \to V$ is a right Lie
$\mathfrak{g}$-module map;

\item[(G2)] $r(x \leftharpoonup g) = [r(x), \, g] +
\bigl(\sigma(x) - x \bigl) \rightharpoonup g$;

\item[(G3)] $\sigma (\{x, y\}) = \{\sigma(x), \, \sigma(y)\} +
\sigma(x) \leftharpoonup r(y) - \sigma(y) \leftharpoonup r(x)$;

\item[(G4)] $r(\{x, \, y\}) = [r(x), \, r(y)] + \sigma(x)
\rightharpoonup r(y) - \sigma(y) \rightharpoonup r(x) + \theta
\bigl(\sigma(x), \sigma(y)\bigl) - \theta(x, y)$
\end{enumerate}
The bijection is such that $\psi = \psi_{(\sigma, r)} \in {\rm
Gal} \, ( \mathfrak{g} \,\natural \, V /\mathfrak{g}) $
corresponding to $(\sigma, r) \in {\rm GL}_k (V) \times {\rm
Hom}_k (V, \, \mathfrak{g})$ is given by $\psi(g, \, x) := (g +
r(x), \, \sigma(x))$, for all $g \in \mathfrak{g}$ and $x \in V$.
We point out that $\psi_{(\sigma, r)}$ is indeed an element of
${\rm Gal} \, ( \mathfrak{g} \,\natural \, V /\mathfrak{g})$ with
the inverse given by $ \psi_{(\sigma, r)}^{-1}(g, \, x) = \bigl(g
- r(\sigma^{-1}(x)), \, \sigma^{-1}(x)\bigl)$, for all $g \in
\mathfrak{g}$ and $x \in V$. We denote by ${\mathbb
G}_{\mathfrak{g}}^V \, \bigl( \leftharpoonup, \, \rightharpoonup,
\, \theta, \{-, \, -\} \bigl)$ the set of all pairs $(\sigma, \,
r) \in {\rm GL}_k (V) \times {\rm Hom}_k (V, \, \mathfrak{g})$
satisfying the compatibility conditions (G1)$-$(G4). It is
straightforward to see that ${\mathbb G}_{\mathfrak{g}}^V \,
\bigl( \leftharpoonup, \, \rightharpoonup, \, \theta, \{-, \, -\}
\bigl)$ is a subgroup of the semidirect product of groups
${\mathbb G}_{\mathfrak{g}}^V := {\rm GL}_k (V) \rtimes {\rm
Hom}_k (V, \, \mathfrak{g})$ with the group structure given by
formula (\ref{grupstrc}). Now, for any $(\sigma, \, r)$ and
$(\sigma', \, r') \in {\mathbb G}_{\mathfrak{g}}^V \, \bigl(
\leftharpoonup, \, \rightharpoonup, \, \theta, \{-, \, -\}
\bigl)$, $g \in \mathfrak{g}$ and $x\in V$ we have:
$$
\psi_{(\sigma, \, r)} \circ \psi_{(\sigma', \, r')} (g, \, x) =
\bigl(g + r'(x) + r(\sigma' (x) ), \, \sigma (\sigma'(x) \bigl) =
\psi_{(\sigma\circ \sigma', \, r\circ \sigma' + r')} (g, \, x)
$$
i.e. $\psi_{(\sigma, \, r)} \circ \psi_{(\sigma', \, r')} =
\psi_{(\sigma\circ \sigma', \, r\circ \sigma' + r')}$. Finally, we
recall that $\mathfrak{h} = \mathfrak{g} + V$ and $\mathfrak{g}
\cap V = \{0\}$, i.e. any element $y \in \mathfrak{h}$ has a
unique decomposition as $y = g + x$, for $g \in \mathfrak{g}$ and
$x\in V = {\rm Ker} (p)$. Putting all together we proved the
following:

\begin{theorem}\label{grupulgal}
Let $\mathfrak{g} \subseteq \mathfrak{h}$ be an extension of Lie
algebras, $p: \mathfrak{h} \to \mathfrak{g}$ a linear retraction
of the inclusion $\mathfrak{g} \subseteq \mathfrak{h}$, $V = {\rm
Ker} (p)$ and consider $\Lambda (\mathfrak{g}, \, V) =
\bigl(\leftharpoonup, \, \rightharpoonup, \, \theta, \{-, \, -\}
\bigl)$ to be the canonical Lie extending system associated to
$p$. Then there exists an isomorphism of groups defined for any
$(\sigma, \, r)\in {\mathbb G}_{\mathfrak{g}}^V \, \bigl(
\leftharpoonup, \, \rightharpoonup, \, \theta, \{-, \, -\}
\bigl)$, $g\in \mathfrak{g}$ and $x\in V$ by:
\begin{equation}\label{izogal}
\Omega: {\mathbb G}_{\mathfrak{g}}^V \, \bigl( \leftharpoonup, \,
\rightharpoonup, \, \theta, \{-, \, -\} \bigl) \, \to {\rm Gal} \,
(\mathfrak{h}/\mathfrak{g}), \quad \Omega (\sigma, r) (g + x) := g
+ r(x) + \sigma (x)
\end{equation}
In particular, there exists an embedding ${\rm Gal} \,
(\mathfrak{h}/\mathfrak{g}) \hookrightarrow {\rm GL}_k (V) \rtimes
{\rm Hom}_k (V, \, \mathfrak{g})$, where the right hand side is
the semidirect product associated to the canonical right action of
${\rm GL}_k (V)$ on ${\rm Hom}_k (V, \, \mathfrak{g})$.
\end{theorem}
In the finite dimensional case we obtain the Lie algebra
counterpart of the fact that the Galois group of a Galois
extension of fields of degree $m$ embeds in the symmetric group
$S_m$.

\begin{corollary}\label{galfinit}
Let $\mathfrak{g} \subseteq \mathfrak{h}$ be an extension of Lie
algebras such that ${\rm dim}_k (\mathfrak{g}) = n$ and ${\rm
dim}_k (\mathfrak{h}) = n + m$. Then the Galois group ${\rm Gal}
\, (\mathfrak{h}/\mathfrak{g})$ embeds in the canonical semidirect
product of groups ${\rm GL} (m, \, k) \rtimes {\rm M}_{n\times m}
(k)$.
\end{corollary}

The first examples based on Theorem~\ref{grupulgal} are given
below. More examples and applications will be presented in
Section~\ref{exempleconc}.

\begin{example}\label{priexe}
Consider the extension of Lie algebras $ke_3 \subseteq
\mathfrak{sl} (2, k)$ with the notations of Example~\ref{exgract}
and take $p: \mathfrak{sl} (2, k) \to ke_3$ given by $p(e_1) =
p(e_2) := 0$ and $p (e_3) := e_3$. Then $V = {\rm Ker } (p) = ke_1
+ ke_2$ and the canonical Lie extending system associated to $p$
given by (\ref{can1})$-$(\ref{can2}) takes the following form:
$\rightharpoonup : V \times ke_3 \to V$ and $\{-, - \} : V \times
V \to V$ are both trivial maps, while the action $\leftharpoonup :
V\times ke_3 \to V$ and the cocycle $\theta : V\times V \to ke_3$
are given by $e_1 \leftharpoonup e_3 = -2e_1$, $e_2
\leftharpoonup e_3 = 2e_2$ and $\theta (e_1, \, e_1) = \theta
(e_2, \, e_2) = 0$, $\theta (e_1, \, e_2) = - \theta (e_2, \, e_1)
= e_3$. In particular, this shows that $\mathfrak{sl} (2, k)$ is
isomorphic as a Lie algebra to the skew crossed product $ke_3
\#^{\bullet}\, V$ of the abelian Lie algebras of dimensions one and
two. Now, an element $\sigma \in {\rm GL}_k (V)$ will be written
in the matrix form $\sigma = (\sigma_{ij}) \in {\rm M}_2 (k)$ and
a linear map $r \in {\rm Hom}_k (V, \, ke_3)$ as a family of two
scalars $r = (r_1, \, r_2) \in k^2$ given by $r (e_1) = r_1 e_3$
and $r (e_2) = r_2 e_3$. A straightforward computation proves that
the pair $\bigl( \sigma = (\sigma_{ij}), \, r = (r_1, \, r_2)
\bigl)$ satisfies (G1)$-$(G4) if and only if $r_1 = r_2 = 0$,
$\sigma_{12} = \sigma_{21} = 0$ and $\sigma_{11} \sigma_{22} = 1$.
This proves that the group ${\mathbb G}_{ke_3}^V \, \bigl(
\leftharpoonup, \, \theta\bigl)$ identifies with the group of
units $k^*$ and hence ${\rm Gal} \, (\mathfrak{sl} (2, k)/ke_3)
\cong k^*$. More precisely, $\tau \in {\rm Gal} \, (\mathfrak{sl}
(2, k)/ke_3)$ if and only if there exists $u\in k^*$ such that
$\tau (a e_1 + b e_2 + c e_3) = u a e_1 + u^{-1} b e_2 + c e_3$,
for all $a$, $b$, $c \in k$.
\end{example}

Our next example proves that the Galois group of the extension of
two consecutive Lie Heisenberg algebras is the $2$-dimensional
special affine group ${\rm SL}_{2}(k) \rtimes k^2$.

\begin{example}\label{heis}
Let $n \in \mathbb N^{*}$ and consider $\mathfrak{h}^{2n+1}$ to be
the $(2n+1)$-dimensional Heisenberg Lie algebra having $\{x_{1},
\cdots, x_{n}, \, y_{1}, \cdots, y_{n}, \, w \}$ as a basis and
the bracket given by $[x_{i}, \,y_{i}] = w$, for all $i = 1,
\cdots, n$. If we consider the canonical Lie algebra extension $
\mathfrak{h}^{2n+1} \subset \mathfrak{h}^{2n+3}$, then there
exists an isomorphism of groups:
$$
{\rm Gal} \, ( \mathfrak{h}^{2n+3}/\mathfrak{h}^{2n+1} ) \cong
{\rm SL}_{2}(k) \rtimes k^2
$$
where  ${\rm SL}_{2}(k) \rtimes k^2$ is the semidirect product of
groups corresponding to the canonical right action $\triangleleft:
k^2 \times {\rm SL}_{2}(k) \to k^2$ given by $(a, b) \triangleleft
B = (a, b) B$, for all $(a, \, b) \in k^2$, $B \in {\rm
SL}_{2}(k)$. To start with, we point out that $\mathfrak{h}^{2n+3}
$ can be realized as a unified product between
$\mathfrak{h}^{2n+1} $ and the vector space $V$ with $k$ basis
$\{x_{n+1},\, y_{n+1}\}$ corresponding to the Lie extending system
with one non-trivial map, namely $\theta : V\times V \to
\mathfrak{h}^{2n+1}$ given by $\theta(x_{n+1}, \, y_{n+1}) = w$.
The conclusion now follows by applying Theorem~\ref{grupulgal}.
First notice that any pair $(\sigma, \, r) \in {\rm GL}_k (V)
\times {\rm Hom}_k (V, \, \mathfrak{g})$ fulfills trivially the
compatibility conditions (G1) and (G3). Furthermore, the
compatibility condition (G2) yields $r(x_{n+1}) = \alpha\, w$ and
$r(y_{n+1}) = \beta\, w$ for some $\alpha$, $\beta \in k$.
Finally, if we denote $\sigma(x_{n+1}) = a x_{n+1} + b x_{n+1}$
and respectively $\sigma(y_{n+1}) = c x_{n+1} + d x_{n+1}$ for
some $a$, $b$, $c$, $d \in k$, the compatibility condition (G4)
gives $ad-bc=1$. Therefore, the set of pairs $(\sigma, \, r) \in
{\rm GL}_k (V) \times {\rm Hom}_k (V, \, \mathfrak{g})$ satisfying
(G1)$-$(G4) is in fact equal to ${\rm SL}_{2}(k) \times k^2$. The
proof is now finished by identifying the maps $\sigma \in {\rm
GL}_k (V)$, $r \in {\rm Hom}_k (V, \, \mathfrak{g})$ with their
corresponding matrices in ${\rm SL}_{2}(k)$ respectively $k^2$ and
noticing that the multiplication on ${\mathbb
G}_{\mathfrak{h}^{2n+1}}^{V} \, (\theta)$ comes down to that
corresponding to the semidirect product induced by the action
defined above.
\end{example}

Now we provide an example of a Lie algebra extension having a
metabelian Galois group.

\begin{example}\label{sigmaex}
For a positive integer $n$, let $\mathfrak{l}(2n+1)$ be the
metabelian Lie algebra with basis $\{E_{i}, \, F_{i}, \, G ~|~ i =
1, \cdots, n\}$ and bracket given by $[E_{i}, \, G] = E_{i}$, $[G,
F_{i}] = F_{i}$, for all $i = 1, \cdots, n$. Then there exists an
isomorphism of groups:
$$
{\rm Gal} \, (\mathfrak{l}(2n+3) /\mathfrak{l}(2n+1)) \cong
(k^{*}\times k^{*}) \rtimes \mathcal{M}_{n\times 2}(k)
$$
where  $(k^{*}\times k^{*}) \rtimes \mathcal{M}_{n\times 2}(k)$ is
the semidirect product corresponding to the right action
$\triangleleft: \mathcal{M}_{n\times 2}(k) \times (k^{*}\times
k^{*}) \to \mathcal{M}_{n \times 2}(k)$ given by $B \triangleleft
(a\,b) = B
\begin{pmatrix} a & 0\\ 0 & b \end{pmatrix}$, for all $a$, $b \in
k^{*}$ and $B \in \mathcal{M}_{n \times 2}(k)$. Indeed,
$\mathfrak{l}(2n+3)$ can be written as a unified product between
$\mathfrak{l}(2n+1)$ and the vector space $V$ with basis
$\{E_{n+1},\, F_{n+1}\}$ corresponding to the extending structure
with only one non-trivial map, namely $\rightharpoonup: V \times
\mathfrak{l}(2n+1) \to V$ given by $E_{n+1} \rightharpoonup G =
E_{n+1}$ and $F_{n+1} \rightharpoonup G = -F_{n+1}$. Now if
$(\sigma, \, r) \in {\rm GL}_k (V) \times {\rm Hom}_k (V, \,
\mathfrak{l}(2n+1))$ a careful analysis of the compatibility
conditions (G1)$-$(G4) yields:
\begin{eqnarray*}
\sigma(E_{n+1}) = aE_{n+1},&& \sigma(F_{n+1}) = bF_{n+1},\,\, a,\, b \in k,\,\, ab \neq 0\\
r(E_{n+1}) = \sum_{i=1}^{n} \alpha_{i} \, E_{i},&& r(F_{n+1}) =
\sum_{i=1}^{n} \beta_{i} \, F_{i},\,\, \alpha_{i}, \beta_{i} \in
k,\, i = 1, 2, \cdots n.
\end{eqnarray*}
Thus, the pairs $(\sigma, \, r) \in {\rm GL}_k (V) \times {\rm
Hom}_k (V, \, \mathfrak{l}(2n+1))$ are parameterized by
$(k^{*}\times k^{*}) \times \mathcal{M}_{n\times 2}(k)$ and the
conclusion follows by Theorem~\ref{grupulgal}.
\end{example}

In the sequel we consider three general examples.

\begin{example}\label{semidirecgal}
Let $\mathfrak{g} \rtimes \mathfrak{h}$ be a semidirect product of
two Lie algebras $\mathfrak{g}$ and $\mathfrak{h}$ written in the
right hand side convention as indicated in
Example~\ref{produssucit}: thus, $\mathfrak{g} \rtimes
\mathfrak{h}$ is associated to a right action of $\mathfrak{g}$ on
$\mathfrak{h}$ denoted by $\leftharpoonup:  \mathfrak{h} \times
\mathfrak{g} \to \mathfrak{h}$. If $\mathfrak{g}$ is abelian and
$\mathfrak{h}$ is a perfect Lie algebra (i.e. $\mathfrak{h} =
[\mathfrak{h}, \, \mathfrak{h}]$), then there exists an
isomorphism of groups:
$$
{\rm Gal} \, (\mathfrak{g} \rtimes \mathfrak{h}/\mathfrak{g})
\cong {\rm Aut}_{\rm Lie}^{\leftharpoonup}(\mathfrak{h})
$$
where ${\rm Aut}_{\rm Lie}^{\leftharpoonup}(\mathfrak{h})$ denotes
the set of Lie algebra automorphisms of $\mathfrak{h}$ which are
also right Lie $\mathfrak{g}$-module maps, i.e. ${\rm Aut}_{\rm
Lie}^{\leftharpoonup}(\mathfrak{h}) = \{u \in {\rm Aut}_{\rm
Lie}(\mathfrak{h}) ~|~ u(x \leftharpoonup g) = u(x) \leftharpoonup
g \, {\rm for}\, {\rm all}\,  g \in \mathfrak{h}, \, x \in
\mathfrak{h}\}$. We just apply Theorem~\ref{grupulgal}. Indeed,
let $(\sigma, \, r) \in {\rm GL}_k (V) \times {\rm Hom}_k (V, \,
\mathfrak{g})$ satisfying the compatibility conditions
(G1)$-$(G4). As $\mathfrak{g}$ is abelian (G2) comes down to
$r\bigl(\{x, \, y\}\bigl) = 0$ for all $x$, $y \in \mathfrak{h}$.
Now since $\mathfrak{h}$ is a perfect Lie algebra we obtain $r=0$.
Then (G2) is trivially fulfilled while (G1) and (G3) imply that
$\sigma$ is a right Lie $\mathfrak{g}$-module map respectively a
Lie algebra map.
\end{example}

The next example computes the Galois group of the extension
$\mathfrak{g}' \subset \mathfrak{g}$ for a special class of Lie
algebras $\mathfrak{g}$, namely the non-perfect ones with
$C_{\mathfrak{g}}(\mathfrak{g}') = \{0\}$, where $\mathfrak{g}' =
[\mathfrak{g}, \, \mathfrak{g}]$ is the derived algebra of
$\mathfrak{g}$ and $C_{\mathfrak{g}}(\mathfrak{g}')$ denotes the
centralizer of $\mathfrak{g}'$ in $\mathfrak{g}$. A generic
example of such a Lie algebra is for instance $\mathfrak{g} :=
\mathfrak{g}\mathfrak{l} (n, \, k) \rtimes k^{n}$, the semidirect
product of Lie algebras corresponding to the canonical action of
$\mathfrak{g}\mathfrak{l} (n, \, k)$ on $k^{n}$.

\begin{example}
Let $\mathfrak{g}$ be a non-perfect Lie algebra such that
$C_{\mathfrak{g}}(\mathfrak{g}') = \{0\}$. Then there exists an
isomorphism of groups
$$
{\rm Gal} \, (\mathfrak{g} /\mathfrak{g}') \cong {\rm GL}_k (V)
$$
where $V$ is a complement as vector spaces of $\mathfrak{g}'$ in
$\mathfrak{g}$. First we write $\mathfrak{g}$ as the unified product between
$\mathfrak{g}'$ and $V$ associated to the Lie extending system
whose non-trivial maps are given as follows: $x \rightharpoonup g
= [x,\,g]$ and $\theta(x,\, y) = [x,\,y]$, for all $g \in
\mathfrak{g}'$, $x$, $y \in V$. Now let $(\sigma, \, r) \in {\rm
GL}_k (V) \times {\rm Hom}_k (V, \, \mathfrak{g})$ satisfying the
compatibility conditions (G1)$-$(G4). One can easily see that (G1)
and (G3) are trivially fulfilled while (G2) and (G4) come down to
the following compatibilities:
\begin{eqnarray*}
\left[r(x) + \sigma(x) - x,\, g\right] = 0, \qquad \left[r(x) +
\sigma(x), \, r(y) + \sigma(y) \right] = \left[x, \, y \right]
\end{eqnarray*}
for all $g \in \mathfrak{g}'$, $x$, $y \in V$. We obtain that
$r(x) + \sigma(x) - x \in C_{\mathfrak{g}}(\mathfrak{g}') =
\{0\}$, for all $x \in V$. Thus $r = {\rm Id}_{V} - \sigma$ and
hence the second equation is now trivially fulfilled. Therefore,
the pairs $(\sigma, \, r) \in {\rm GL}_k (V) \times {\rm Hom}_k
(V, \, \mathfrak{g})$ satisfying the compatibility conditions
(G1)$-$(G4) are of the form $(\sigma, \, {\rm Id}_{V} - \sigma)$
with $\sigma \in {\rm GL}_k (V)$. In this case the multiplication
given by equation~(\ref{grupstrc}) becomes $(\sigma, \, {\rm Id}_{V}
- \sigma) \cdot (\sigma', \, {\rm Id}_{V} - \sigma') = (\sigma
\circ \sigma ',\, {\rm Id}_{V} - \sigma \circ \sigma ')$ and thus
${\mathbb G}_{\mathfrak{g'}}^{V} \, (\rightharpoonup,\, \theta)$
is isomorphic to ${\rm GL}_k (V)$. The conclusion now follows from
Theorem~\ref{grupulgal}.
\end{example}

Let $\mathfrak{H}(\mathfrak{g})$ be the \textit{holomorph}
\cite{SZ} Lie algebra of a Lie algebra $\mathfrak{g}$, i.e.
$\mathfrak{H}(\mathfrak{g}) = \mathfrak{g} \times {\rm
Der}(\mathfrak{g})$ endowed with the Lie bracket given by:
$\left[(g, \, \varphi),\, (h, \, \psi)\right] = \big\{[g, \, h] +
\varphi(h) - \psi(g), \, [\varphi, \, \psi] \big\}$, for all $g$,
$h \in \mathfrak{g}$ and $\varphi$, $\psi \in {\rm
Der}(\mathfrak{g})$.

\begin{example}
Let $\mathfrak{g}$ be a complete Lie algebra (i.e. $\mathfrak{g}$ has trivial center and only inner derivations; see \cite{jacobson} for further details). Then
there exists an isomorphism of groups:
$$
{\rm Gal} \, (\mathfrak{H}(\mathfrak{g}) /\mathfrak{g}) \cong {\rm
Aut}_{{\rm Lie}} (\mathfrak{g})
$$
To start with we point out that since $\mathfrak{g}$ is complete
all derivations are inner, i.e. ${\rm Der}(\mathfrak{g}) = \{{\rm
ad}_{x} ~|~ x \in \mathfrak{g}\}$. It can be easily seen that
$\mathfrak{h}(\mathfrak{g})$ is a unified product between
$\mathfrak{g}$ and ${\rm Der}(\mathfrak{g})$ corresponding to the
extending system whose non-trivial maps are given as follows:
\begin{eqnarray*}
{\rm ad}_{x} \rightharpoonup g = [x,\,g], \qquad \{{\rm ad}_{x},\,
{\rm ad}_{y}\} = {\rm ad}_{[x,\,y]}
\end{eqnarray*}
for all $g$, $x$, $y \in \mathfrak{g}$. Consider now $(\sigma, \,
r) \in {\rm GL}_k ({\rm Der} (\mathfrak{g})) \times {\rm Hom}_k
({\rm Der} (\mathfrak{g}), \, \mathfrak{g})$ satisfying the
compatibility conditions (G1)$-$(G4). As for any $x \in
\mathfrak{g}$ we have $\sigma({\rm ad}_{x}) \in {\rm Der}
(\mathfrak{g})$ it follows that $\sigma({\rm ad}_{x}) = {\rm
ad}_{\tau(x)}$ for some bijective linear map $\tau : \mathfrak{g}
\to \mathfrak{g}$. One can easily see that (G1) is trivially
fulfilled, (G3) comes down to $\tau$ being a Lie algebra map
while(G2) yields:
\begin{eqnarray*}
\left[r(ad_{x}) + \tau(x) - x, g \right] = 0, \,\, {\rm for}\,\,
{\rm all}\,\, x,\, g \in \mathfrak{g}
\end{eqnarray*}
Hence $r(ad_{x}) + \tau(x) - x \in Z(\mathfrak{g})$ for all $x \in
\mathfrak{g} $, where $Z(\mathfrak{g})$ denotes the center of the
Lie algebra $\mathfrak{g}$; as $\mathfrak{g}$ is complete
we obtain $r({\rm ad}_{x}) = x -
\tau(x)$ for all $x \in \mathfrak{g}$. Under these assumptions
(G4) is also trivially fulfilled. To summarize, we proved that any
pair $(\sigma, \, r) \in {\rm GL}_k ({\rm Der} (\mathfrak{g}))
\times {\rm Hom}_k ({\rm Der} (\mathfrak{g}), \, \mathfrak{g})$
satisfying the compatibility conditions (G1)$-$(G4) is implemented
by a Lie algebra automorphism $\tau : \mathfrak{g} \to
\mathfrak{g}$ as follows:
\begin{eqnarray*}
\sigma({\rm ad}_{x}) = {\rm ad}_{\tau(x)}, \quad r({\rm ad}_{x}) =
x - \tau(x)
\end{eqnarray*}
for all $x \in \mathfrak{g}$. An easy computation shows that the
map which sends each pair $(\sigma, \, r) \in {\rm GL}_k ({\rm
Der} (\mathfrak{g})) \times {\rm Hom}_k ({\rm Der} (\mathfrak{g}),
\, \mathfrak{g})$ to the corresponding Lie algebra automorphism
$\tau : \mathfrak{g} \to \mathfrak{g}$ is a group automorphism
between ${\mathbb G}_{\mathfrak{g}}^{{\rm Der} (\mathfrak{g})} \,
(\rightharpoonup,\, \theta)$ and ${\rm Aut}_{{\rm Lie}}
(\mathfrak{g})$. Now Theorem~\ref{grupulgal} is the last step in
obtaining the desired conclusion.
\end{example}

We now specialize the discussion to extensions of the form
$\mathfrak{h}^G \subseteq \mathfrak{h}$, where $G$ is a group
acting on a Lie algebra $\mathfrak{h}$. Our approach has Artin's theorem \cite[Theorem 1.8]{Lang} as source
of inspiration: if $ G
\leq {\rm Aut} (K)$ is a finite group of automorphisms of a field
$K$, then $K$ is isomorphic to a crossed product algebra $k \#_{\sigma} \, k[G]^*$
between the field of invariants $k = K^G$ and the dual algebra of
the group algebra $k[G]$, associated to some cocycle $\sigma : k[G]^* \otimes k[G]^* \to k$. In what follows we will prove the Lie
algebra counterpart of this very important result. Let $G$ be a
finite group whose order $|G|$ is invertible in $k$ and suppose
$G$ is acting on the Lie algebra $\mathfrak{h}$ via the group morphism $\varphi:
G \to {\rm Aut}_{\rm Lie} (\mathfrak{h})$, $\varphi (g) (x) = g
\triangleright x$, for all $g\in G$ and $x\in \mathfrak{h}$. Our
goal is to describe the Galois group ${\rm Gal} \,
(\mathfrak{h}/\mathfrak{h}^G)$ and to rebuild $\mathfrak{h}$ from
the subalgebra of invariants $\mathfrak{h}^G$ and an extra set of
data. We mention that if $G \leq {\rm Aut}_{\rm Lie}
(\mathfrak{h})$ is a subgroup of the Lie algebra automorphism of
$\mathfrak{h}$ acting on $\mathfrak{h}$ via the canonical action
$\sigma \triangleright x : = \sigma (x)$, for all $\sigma \in G
\leq {\rm Aut}_{\rm Lie} (\mathfrak{h})$, then we have $G
\subseteq {\rm Gal} \, (\mathfrak{h}/\mathfrak{h}^G)$ - as opposed
to the classical Artin's theorem, we will se that for Lie algebras
we are far from having equality in the inclusion $G \subseteq {\rm
Gal} \, (\mathfrak{h}/\mathfrak{h}^G)$. Since $|G|$ is invertible
in $k$, we can choose  the trace map $t : \, \mathfrak{h} \to
\mathfrak{h}^G$ defined by $t (x) := |G|^{-1} \, \sum_{\gamma \in
G} \, \gamma \triangleright x$, for all $x\in \mathfrak{h}$ as a
linear retraction of the inclusion $\mathfrak{h}^G \hookrightarrow
\mathfrak{h}$. We shall compute the canonical extending system of
$\mathfrak{h}^G$ through $V : = {\rm Ker} (t)$ associated to the
trace map $t$ using the formulas (\ref{can1})$-$(\ref{can2}). For any
$x\in V$ and $g\in \mathfrak{h}^G$ we have:
\begin{eqnarray*}
x \rightharpoonup g &=& t ([x, \, g]) = |G|^{-1} \, \sum_{\gamma
\in G} \, \gamma \triangleright [x, \, g] = |G|^{-1} \,
\sum_{\gamma \in G} \, [\gamma \triangleright x, \, \gamma
\triangleright g] \\
&=& |G|^{-1} \, \sum_{\gamma \in G} \, [\gamma \triangleright x,
\, g] = [t(x), \, g] = 0
\end{eqnarray*}
where the equalities in the last line follow from $g \in \mathfrak{h}^G$
and $x\in V = {\rm Ker} (t)$, respectively. Moreover, we can easily see that the
action $\leftharpoonup $, the cocycle $\theta$ and the
quasi-bracket $\{-, \, -\}$ on $V$ take the form:
\begin{eqnarray}
&& x \leftharpoonup g = [x, \, g], \quad \theta (x, \, y) =
|G|^{-1} \, \sum_{\gamma \in G} \, [\gamma \triangleright x, \,
\gamma \triangleright y] \label{formagract}\\
&& \{x, \, y\} = [x, \, y] - |G|^{-1} \, \sum_{\gamma \in G} \,
[\gamma \triangleright x, \, \gamma \triangleright y]
\label{formagractb}
\end{eqnarray}
for all $x$, $y\in V$ and $g \in \mathfrak{h}^G$. The left action
$\rightharpoonup$ being the trivial map has an important
consequence: using Example~\ref{produssucit} it follows that the
unified product $\mathfrak{h}^G \, \natural \, V $ associated to
this canonical extending system of $\mathfrak{h}^G$ by $V$ reduces
to a skew crossed product $\mathfrak{h}^G \, \#^{\bullet} \, V $
and the map defined for any $g\in \mathfrak{h}^G$ and $x\in V$ by:
\begin{equation}\label{recizo}
\varphi: \mathfrak{h}^G \, \#^{\bullet} \, V \to \mathfrak{h},
\qquad \varphi(g, x) := g + x
\end{equation}
is an isomorphism of Lie algebras. The Lie bracket on
$\mathfrak{h}^G \, \#^{\bullet} \, V$ given by
equation~(\ref{brackskcp}) takes the following form:
\begin{eqnarray}
[(g, \, x), \, (g', \, x')] &:=& \Bigl( [g, \, g'] + |G|^{-1} \,
\sum_{\gamma \in G} \, [\gamma \triangleright x, \, \gamma
\triangleright x'], \nonumber \\
&& [x, \, x'] - |G|^{-1} \, \sum_{\gamma \in G} \, [\gamma
\triangleright x, \, \gamma \triangleright x'] + [x, \, g'] - [x',
\, g] \Bigl)\label{bracsc}
\end{eqnarray}
for all $g$, $g' \in \mathfrak{h}^G$ and $x$, $x' \in V$. Given a
group $G$ acting on a Lie algebra $\mathfrak{h}$, the isomorphism
given in equation~(\ref{recizo}) provides the reconstruction of
$\mathfrak{h}$ from the subalgebra of invariants $\mathfrak{h}^G$.
We continue our investigation in order to describe the Galois
group ${\rm Gal} \, (\mathfrak{h}/\mathfrak{h}^G)$. Since the
components of the canonical extending system given by equations~
(\ref{formagract})$-$(\ref{formagractb}) are implemented only by the
action $\varphi$ of $G$ on $\mathfrak{h}$ we shall denote the
group ${\mathbb G}_{\mathfrak{h}^G}^V \, \bigl( \leftharpoonup, \,
\rightharpoonup, \, \theta, \{-, \, -\} \bigl)$ constructed in
Theorem~\ref{grupulgal} by ${\mathbb G}_{\mathfrak{h}^G}^V \,
\bigl( \varphi \bigl)$. Thus ${\mathbb G}_{\mathfrak{h}^G}^V \,
\bigl( \varphi \bigl)$ consist of the set of all pairs $(\sigma,
\, r) \in {\rm GL}_k (V) \times {\rm Hom}_k (V, \,
\mathfrak{h}^G)$ satisfying the following compatibility conditions
for all $g \in \mathfrak{h}^G$ and $x$, $y\in V$:
\begin{eqnarray*}
&& \sigma ([x, \, g]) = [\sigma(x), \, g], \qquad r ([x, \, g]) =
[r(x), \, g] \\
&& \sigma ([x, \, y]) - [\sigma(x), \, \sigma(y) ] = [\sigma(x),
\, r(y) ] + [r(x), \, \sigma(y) ] + \\
&& + |G|^{-1} \, \sum_{\gamma \in G} \, \sigma \bigl( [\gamma
\triangleright x, \, \gamma \triangleright y] \bigl) - |G|^{-1} \,
\sum_{\gamma \in G} \, [\gamma \triangleright \sigma(x), \, \gamma \triangleright \sigma(y)] \\
&& r ([x, \, y]) - [r(x), \, r(y) ] =  |G|^{-1} \, \sum_{\gamma
\in G} \, r \bigl( [\gamma \triangleright x, \, \gamma \triangleright y] \bigl) + \\
&& + |G|^{-1} \, \sum_{\gamma \in G} \, [\gamma \triangleright
\sigma(x), \, \gamma \triangleright \sigma(y)] - |G|^{-1} \,
\sum_{\gamma \in G} \, [\gamma \triangleright x, \, \gamma
\triangleright y]
\end{eqnarray*}
which is exactly what is left from axioms (G1)$-$(G4) after using
equations~(\ref{formagract})$-$(\ref{formagractb}) and the fact that
$\rightharpoonup$ is the trivial action. We note that the first
two compatibilities above show that $\sigma$ and $r$ are morphisms
of right Lie $\mathfrak{h}^G$-modules while the last two
compatibilities measures how far they are from being Lie algebra
maps. ${\mathbb G}_{\mathfrak{h}^G}^V \, \bigl( \varphi \bigl)$ is
a group with multiplication given by equation~(\ref{grupstrc}).
We record all this facts in the following:

\begin{theorem} \textbf{(Artin's Theorem for Lie algebras)}\label{recsiGal}
Let $G$ be a finite group of invertible order in $k$ acting on a
Lie algebra $\mathfrak{h}$ via $\varphi: G \to {\rm Aut}_{\rm Lie}
(\mathfrak{h})$. Let $\mathfrak{h}^G \subseteq \mathfrak{h}$ be
the subalgebra of invariants and $V = {\rm Ker} (t)$, where $t:
\mathfrak{h} \to \mathfrak{h}^G$ is the trace map. Then:

(1) The map defined for any $g\in \mathfrak{h}^G$ and $x\in V$ by:
\begin{equation}\label{recizoaa}
\varphi: \mathfrak{h}^G \, \#^{\bullet} \, V \to \mathfrak{h},
\qquad \varphi(g, x) := g + x
\end{equation}
is an isomorphism of Lie algebras, where $\mathfrak{h}^G \, \#^{\bullet} \,
V $ is the skew crossed product of Lie algebras having the bracket
given by (\ref{bracsc}).

(2) The map defined for any $(\sigma, r) \in {\mathbb
G}_{\mathfrak{h}^G}^V \, \bigl( \varphi \bigl)$, $g \in
\mathfrak{h}^G$ and $x\in V$ by:
\begin{equation}\label{izogalaa}
\Omega: {\mathbb G}_{\mathfrak{h}^G}^V \, \bigl( \varphi \bigl) \,
\to {\rm Gal} \, (\mathfrak{h}/\mathfrak{h}^G), \quad \Omega
(\sigma, r) (g + x) := g + r(x) + \sigma (x)
\end{equation}
is an isomorphism of groups.
\end{theorem}
Even if $G$ is a finite subgroup of ${\rm
Aut}_{\rm Lie} (\mathfrak{h})$, Theorem~\ref{recsiGal} $(2)$ shows that we are far away from having ${\rm
Gal} \, (\mathfrak{h}/\mathfrak{h}^G) \cong G$ as in the case of
fields. We present a relevant example below:

\begin{example}\label{exrelmare}
Let $k$ be a field of characteristic $\neq 2$ and $\varphi: k^*
\to {\rm Aut}_{\rm Lie} \bigl( \mathfrak{sl} (2, k) \bigl)$ the
action of $k^*$ on $\mathfrak{sl} (2, k)$ given by
equation~(\ref{sl2}). The subalgebra of invariants $\mathfrak{sl}
(2, k)^{k^*}$ of this action is just $ke_3$ and
Example~\ref{priexe} shows that the Galois group ${\rm Gal} \, (\mathfrak{sl} (2,
k)/ \mathfrak{sl} (2, k)^{k^*})$ is isomorphic to $k^*$. This situation occurs
rarely. Indeed, if we consider $G := U_2 (k) = \{ \pm 1 \}$ the
cyclic subgroup of $k^*$ of roots of unity of order two and the
same action of $U_2 (k) \leq k^*$ on $\mathfrak{sl} (2, k)$ as
above we obtain the same subalgebra of invariants, namely
$\mathfrak{sl} (2, k)^{U_2 (k)} = ke_3$. Hence we have that ${\rm
Gal} \, (\mathfrak{sl} (2, k)/ \mathfrak{sl} (2, k)^{U_2 (k) })
\cong k^* \neq U_2(k)$.
\end{example}

A crucial step in applying Theorem~\ref{recsiGal} is the
description of the kernel of the trace map $t: \mathfrak{h} \to
\mathfrak{h}^G$, which heavily depends of the group $G$ and on the
action $\varphi$. In the case of cyclic groups acting on fields
this kernel is described by Hilbert's theorem \cite[Theorem
6.3]{Lang}. As a nice surprise its counterpart for Lie algebras is
also true but the proof uses completely different techniques.

\begin{theorem}\textbf{(Hilbert's Theorem 90 for Lie algebras)}\label{Hilbert90}
Let $G$ be a finite cyclic group generated by an element $\gamma$
whose order $n$ is invertible in $k$. Let $\varphi: G \to {\rm
Aut}_{\rm Lie} (\mathfrak{h})$ be a morphism of groups and $t :
\mathfrak{h} \to \mathfrak{h}^G$ the trace map. Then $ {\rm Ker}
(t) = \{ y - \gamma \triangleright y \, | \, y \in
\mathfrak{h}\}$.
\end{theorem}

\begin{proof}
It is straightforward to see that $t (y - g \triangleright y) =
0$, for all $y \in \mathfrak{h}$. Conversely, let $a \in {\rm Ker}
(t)$. We define recursively the sequence of elements $(d_i)_{i\geq
0}$ of ${\rm Hom}_k (\mathfrak{h}, \, \mathfrak{h})$ by the
formulas:
$$
d_0 (y) := a + y, \qquad d_{i+1} (y) := a + \gamma \triangleright
d_i(y),
$$
for all $i\geq 0$ and $y \in \mathfrak{h}$. Thus, we have $d_1 (y)
= a + \gamma \triangleright a + \gamma \triangleright y$,
$\cdots$, $d_{n-2} (y) = a + \gamma \triangleright a + \cdots +
\gamma^{n-2} \triangleright a + \gamma^{n-2} \triangleright y$ and
using $t(a) = 0$ we obtain $d_{n-1} (y) = \gamma^{n-1}
\triangleright y$ and hence $d_n (y) = a + \gamma \triangleright (
\gamma^{n-1} \triangleright y ) = a + y = d_0 (y)$. Theferore $d_n
= d_0$ i.e., the sequence $(d_i)_{i\geq 0}$ is periodic. Now, in
the abelian group ${\rm Hom}_k (\mathfrak{h}, \, \mathfrak{h})$ we
add all the equalities listed below:
$$
d_1 = a + \gamma \triangleright d_0, \quad  d_2 = a +
\gamma\triangleright d_1, \quad  \cdots,  d_{n-1} = a + \gamma
\triangleright d_{n-2}, \quad d_n = a + \gamma \triangleright
d_{n-1}
$$
and using $d_n = d_0$, we obtain $\sum_{i=0}^{n-1} \, d_i = n \, a
+ \gamma \triangleright \bigl( \sum_{i=0}^{n-1} d_i \bigl)$. If
$d_0 + d_1 + \cdots d_{n-1} = 0$ in the abelian group ${\rm Hom}_k
(\mathfrak{h}, \, \mathfrak{h})$, we obtain using the
invertibility of $n$ in $k$, that $a = 0 = 0 - \gamma
\triangleright 0$ and we are done. On the other hand, if $d_0 +
d_1 + \cdots d_{n-1} \neq 0$ we can pick some $z \in \mathfrak{h}$
such that $y := \sum_{i=0}^{n-1} \, d_i (z) \neq 0$. Then, using
$d_n(z) = d_0 (z)$, we have:
\begin{eqnarray*}
n \, a + \gamma \triangleright y = n \, a + \sum_{i=0}^{n-1}
\gamma \triangleright d_i (z) = \sum_{i=0}^{n-1} \bigl(a + \gamma
\triangleright d_i (z) \bigl) = \sum_{i=0}^{n-1} d_{i+1} (z) = y
\end{eqnarray*}
This shows that $a = n^{-1} \, \bigl(y - \gamma \triangleright y
\bigl)$ and the proof is finished.
\end{proof}

Now we introduce the following:

\begin{definition}\label{g-abelian}
Let $\mathfrak{h}$ be a Lie algebra, $G$ a group, $\varphi: G \to
{\rm Aut}_{\rm Lie} (\mathfrak{h})$ a morphism of groups, $\gamma
\in G$ and $\mathfrak{h}_{\gamma} := \{ y - \gamma \triangleright
y \, | \, y\in \mathfrak{h}\}$. The action $\varphi$ is called
$\gamma$-abelian if:
\begin{equation}\label{gamab}
[g \triangleright z , \, g'\triangleright z'] = 0
\end{equation}
for all $g \neq g' \in G$ and $z$, $z'\in \mathfrak{h}_{\gamma}$.
\end{definition}

The structure theorem for cyclic Galois extensions of fields
\cite[Theorem 6.2]{Lang} can be rephrased as follows: if $G \leq
{\rm Aut} (K)$ is a cyclic subgroup of order $n$ of the group of
authomorphisms of a field $K$ of characteristic zero and $k :=
K^G$, then $K$ is isomorphic to the splitting field over $k$ of a
polynomial of the form $X^n - a \in k[X]$. The Lie algebra
counterpart of this result now follows by replacing the
splitting field with the semidirect product of Lie algebras:

\begin{corollary}\label{ciclicstru}
Let $\mathfrak{h}$ be a Lie algebra, $G$ a finite cyclic group
generated by an element $\gamma$ whose order $n$ is invertible in
$k$ and $\varphi: G \to {\rm Aut}_{\rm Lie} (\mathfrak{h})$ a
$\gamma$-abelian morphism of groups. Then, the map defined for any
$g\in \mathfrak{h}^G$ and $x\in \mathfrak{h}_{\gamma}$ by:
\begin{equation}\label{recizoaabb}
\varphi: \mathfrak{h}^G \,\rtimes \, \mathfrak{h}_{\gamma} \to
\mathfrak{h}, \qquad \varphi(g, x) := g + x
\end{equation}
is an isomorphism of Lie algebras, where $\mathfrak{h}^G \,
\rtimes \, \mathfrak{h}_{\gamma}$ is the semidirect product of Lie
algebras associated to the right action  $\leftharpoonup :
\mathfrak{h}_{\gamma} \times \mathfrak{h}^G \to
\mathfrak{h}_{\gamma}$, given by $x \leftharpoonup g :=[x, \, g]$.
\end{corollary}

\begin{proof}
Using Theorem~\ref{Hilbert90} together with Theorem~\ref{recsiGal}
we only need to prove that the cocycle $\theta :
\mathfrak{h}_{\gamma} \times \mathfrak{h}_{\gamma} \to
\mathfrak{h}^G$ given by (\ref{formagract}) is the trivial map.
Moreover, in this case it also follows that the bracket $\{-, \,
-\}$ on $\mathfrak{h}_{\gamma}$ depicted in (\ref{formagractb})
coincides with the Lie bracket on $\mathfrak{h}$, i.e. $\{x, \,
y\} = [x, \, y]$, for all $x$, $y \in \mathfrak{h}_{\gamma}$ and
$\mathfrak{h}_{\gamma}$ is an ideal of $\mathfrak{h}$. Indeed, let
$y - \gamma \triangleright y$ and $y' - \gamma \triangleright y'$
be two elements of $\mathfrak{h}_{\gamma}$, for some $y$, $y\in
\mathfrak{h}$. Then we have:
\begin{eqnarray*}
\theta (y - \gamma \triangleright y, \, y' - \gamma \triangleright
y') &=& n^{-1} \, \sum_{\delta \in G} \, [\delta \triangleright (y
- \gamma \triangleright y), \, \delta \triangleright (y' - \gamma
\triangleright y')] \\
&=& n^{-1} \sum_{i=0}^{n-1} \, [\gamma^i \triangleright (y -
\gamma \triangleright y), \, \gamma^i \triangleright (y' - \gamma
\triangleright y')] \\
&=& n^{-1} \, [ \sum_{i=0}^{n-1} \, \gamma^i \triangleright (y -
\gamma \triangleright y) , \, \sum_{i=0}^{n-1} \, \gamma^i
\triangleright (y - \gamma \triangleright y) ] = 0
\end{eqnarray*}
where in the third equality we used the fact that $\varphi$ is a
$\gamma$-abelian action while the final equality holds due to the
following trivial identity: $\sum_{i=0}^{n-1} \, \gamma^i
\triangleright (y - \gamma \triangleright y) = 0$.
\end{proof}

\begin{remark}\label{galcicaa}
Under the assumptions of Corollary~\ref{ciclicstru} we can provide
an simpler description of the Galois group ${\rm Gal} \,
(\mathfrak{h}/\mathfrak{h}^G) \cong {\mathbb
G}_{\mathfrak{h}^G}^{\mathfrak{h}_{\gamma}} \, \bigl( \varphi
\bigl)$. Indeed, in this case the group ${\mathbb
G}_{\mathfrak{h}^G}^{\mathfrak{h}_{\gamma}} \, \bigl( \varphi
\bigl)$ as defined in Theorem~\ref{recsiGal} consists of the set
of all pairs $(\sigma, \, r) \in {\rm GL}_k (\mathfrak{h}_{\gamma}
) \times {\rm Hom}_k (\mathfrak{h}_{\gamma}, \, \mathfrak{h}^G)$
satisfying the following compatibility conditions for any $g \in
\mathfrak{h}^G$ and $x$, $y\in \mathfrak{h}_{\gamma}$:
\begin{eqnarray*}
&& \sigma ([x, \, g]) = [\sigma(x), \, g], \qquad r ([x, \, g]) =
[r(x), \, g], \qquad  r ([x, \, y]) = [r(x), \, r(y) ] \\
&& \sigma ([x, \, y]) - [\sigma(x), \, \sigma(y) ] = [\sigma(x),
\, r(y) ] + [r(x), \, \sigma(y) ]
\end{eqnarray*}
which is a subgroup in the semidirect product of groups ${\rm GL}_k
(\mathfrak{h}_{\gamma} ) \rtimes {\rm Hom}_k
(\mathfrak{h}_{\gamma}, \, \mathfrak{h}^G)$.
\end{remark}

\section{Applications and Examples}\label{exempleconc}

In this section we present some applications as well as explicit
examples of Galois groups. The simplest case is that of extensions
$\mathfrak{g} \subseteq \mathfrak{h}$ for which the codimension of
$\mathfrak{g}$ in $\mathfrak{h}$ is equal to $1$. In this case we
will show that the Galois group ${\rm Gal} \,
(\mathfrak{h}/\mathfrak{g})$ is metabelian. To this end, consider
$\mathfrak{g} \subseteq \mathfrak{h}$ to be an extension of Lie
algebras such that $\mathfrak{g}$ has codimension $1$ in
$\mathfrak{h}$. Thus, we can write $\mathfrak{h} = \mathfrak{g} +
V$, where $V := k x$, for a fixed element $x \in \mathfrak{h}
\setminus \mathfrak{g}$. We choose the map $p$ defined by $p(x) :=
0$ and $p(g) = g$, for all $g \in \mathfrak{g}$, as a retraction
of the inclusion map $\mathfrak{g} \hookrightarrow \mathfrak{h}$ .
Now, the space of all Lie extending systems of $\mathfrak{g}$
through $V = kx$ is parameterized by the set ${\rm TwDer}
(\mathfrak{g})$ of all twisted derivations of $\mathfrak{g}$
\cite[Proposition 4.4]{am-2013}. Recall that a twisted derivation
of a Lie algebra $\mathfrak{g}$ is a pair $(\lambda, \Delta)$
consisting of two linear maps $\lambda : \mathfrak{g} \to k$ and
$\Delta : \mathfrak{g} \to \mathfrak{g}$ such that for any $g$,
$h\in \mathfrak{g}$:
\begin{equation}\label{lambderivari}
\lambda ([g, \, h]) = 0, \quad \,\,  \Delta ([g, \, h]) =  [
\Delta (g), \, h] + [g, \, \Delta (h)] + \lambda(g) \Delta (h) -
\lambda(h) \Delta (g)
\end{equation}
The bijection between the set of all Lie extending structures of
$\mathfrak{g}$ through $V = kx$ and ${\rm TwDer} (\mathfrak{g})$
is given by the two-sided formula:
\begin{equation}\label{extendbij}
x \leftharpoonup g :=: \lambda (g) x, \quad x \rightharpoonup g
:=: \Delta(g), \quad \theta :=: 0, \quad \{-, \, -\} :=: 0
\end{equation}
for all $g \in \mathfrak{g}$. Let  $(\lambda, \Delta) \in {\rm
TwDer} (\mathfrak{g})$ be the twisted derivation associated to the
canonical Lie extending system of $\mathfrak{g}$ through $V$
arising from $p$ via (\ref{extendbij}) and denote by
$\mathfrak{g}_{(\lambda, \Delta)} := \mathfrak{g} \,\natural \,
kx$ the corresponding unified product. For future use, we
mention that the Lie algebra $\mathfrak{g}_{(\lambda, \, \Delta)}$
can be defined (see \cite{am-sigma}) as the vector space $\mathfrak{g}
\times k$ with the bracket given as follows for all $x$, $y \in \mathfrak{g}$
and $a$, $b \in k$:
\begin{eqnarray}
\{(x, a) , \, (y, b) \} := \Bigl([x, \, y] + b \, \Delta(x) - a\,
\Delta(y), \ b \, \lambda(x) - a\, \lambda(y) \Bigl)
\label{exdim300aa}
\end{eqnarray}
Of course $\mathfrak{g}_{(\lambda, \, \Delta)}$ contains
$\mathfrak{g} \cong \mathfrak{g} \times \{0\}$ as a subalgebra of
codimension $1$. We observe that $(\lambda := 0, \, \Delta)
\in {\rm TwDer} (\mathfrak{g})$ if and only if $\Delta \in {\rm
Der} (\mathfrak{g})$ is a classical derivation of $\mathfrak{g}$;
in this case we shall denote $\mathfrak{g}_{(\Delta)} :=
\mathfrak{g}_{(0, \, \Delta)}$, for any $\Delta \in {\rm Der}
(\mathfrak{g})$. As an example, we mention that if $\mathfrak{g}$
is a perfect Lie algebra then the first compatibility condition in
(\ref{lambderivari}) yields ${\rm TwDer} (\mathfrak{g}) = \{0\}
\times {\rm Der} (\mathfrak{g})$.

Continuing our investigation it follows that the Galois group
${\rm Gal} \, (\mathfrak{h}/\mathfrak{g}) \cong {\rm Gal} \,
(\mathfrak{g}_{(\lambda, \Delta)}/\mathfrak{g})$, which is a
special case of the isomorphism given by (\ref{izogalgen}). We
denote by ${\mathbb G}_{\mathfrak{g}} \, (\lambda, \, \Delta)$ the
set of all pairs $(u, \, g_0) \in k^* \times \mathfrak{g}$
satisfying the following compatibility condition for all $g \in
\mathfrak{g}$:
\begin{equation}\label{grciud}
\lambda (g) \, g_0 = [g_0, \, g] + (u-1) \, \Delta (g)
\end{equation}
Then ${\mathbb G}_{\mathfrak{g}} \, (\lambda, \, \Delta)$ is a
subgroup in the metabelian group ${\mathbb G}_{\mathfrak{g}} :=
k^* \rtimes \mathfrak{g}$ whose multiplication is given by
equation~(\ref{grupstrcb}), that is $(u, \, g) \cdot (u', \, g') :=
(uu', \, u'g + g')$, for all $u$, $u'\in k^*$ and $g$, $g' \in
\mathfrak{g}$. We can now prove the following:

\begin{corollary}\label{galcodim1}
Let $\mathfrak{g} \subseteq \mathfrak{h}$ be a Lie subalgebra of
codimension $1$ in $\mathfrak{h}$ and $(\lambda, \Delta) \in {\rm
TwDer} (\mathfrak{g})$ the twisted derivation defined by
equation~(\ref{extendbij}) for a fixed $x \in \mathfrak{h} \setminus
\mathfrak{g}$. Then there exists an isomorphism of groups given
for any $(u, \, g_0) \in {\mathbb G}_{\mathfrak{g}} \, (\lambda,
\, \Delta)$, $g\in \mathfrak{g}$ and $\alpha \in k$ by:
\begin{equation}\label{izogalcodim1}
\Omega:  {\mathbb G}_{\mathfrak{g}} \, (\lambda, \, \Delta) \, \to
{\rm Gal} \, (\mathfrak{h}/\mathfrak{g}), \quad \Omega (u, \, g_0)
(g + \alpha \, x) := g + \alpha \, g_0 + u \alpha \, x
\end{equation}
In particular, the Galois group ${\rm Gal} \,
(\mathfrak{h}/\mathfrak{g})$ is metabelian and hence solvable.
\end{corollary}

\begin{proof}
We apply Theorem~\ref{grupulgal}: since $V = kx$, any linear
automorphism $\sigma : V \to V$ is uniquely determined by an
invertible element $u\in k^*$ via $\sigma (x) := u x$ while a
linear map $r: V \to \mathfrak{g}$ is implemented by an element
$g_0 \in \mathfrak{g}$ via $r(x) := g_0$. Now, the axioms (G1),
(G3) and (G4) which define the group ${\mathbb G}_{\mathfrak{g}}^V \,
\bigl( \leftharpoonup, \, \rightharpoonup, \, \theta, \{-, \, -\}
\bigl)$ from Theorem~\ref{grupulgal} are trivially fulfilled,
while axiom (G2) comes down to the compatibility condition
(\ref{grciud}). Finally, the group ${\rm Gal} \,
(\mathfrak{h}/\mathfrak{g})$ is metabelian due to its embedding in
the metabelian group $k^* \rtimes \mathfrak{g}$.
\end{proof}

\begin{example}\label{codim1ana}
Let $n \in {\mathbb N}^{*}$ be a positive integer and consider the
extension of Lie algebras $\mathfrak{h}^{2n+1} \subseteq
\mathfrak{t}^{2n+2}$, where $\mathfrak{h}^{2n+1}$ is the
$(2n+1)$-dimensional Heisenberg Lie algebra from
Example~\ref{heis} and $\mathfrak{t}^{2n+2}$ is the Lie algebra
with basis $\{x_{1}, \cdots, x_{n}, \,  y_{1}, \cdots, y_{n}, \,
w, \, u\}$ and bracket given for all $i = 1, \cdots, n$ by: $[x_i,
\, y_i] = w$, $[u, \, x_{i}] = w + u$, $[u, \, y_{i}] = w + u$.
Then there exists an isomorphism of groups:
$$
{\rm Gal} \, ( \mathfrak{t}^{2n+2}/\mathfrak{h}^{2n+1} ) \cong
(k^{*}, \cdot).
$$
First observe that the Lie algebra $\mathfrak{t}^{2n+2}$ is
isomorphic to $\mathfrak{h}^{2n+1}_{(\lambda, \Delta)}$, where the
twisted derivation $(\lambda, \Delta)$ of the Heisenberg Lie
algebra $\mathfrak{h}^{2n+1}$ is given by: $\lambda(w) := 0$,
$\lambda(x_{i}) = \lambda(y_{i}) := 1$, $\Delta(w) := 0$,
$\Delta(x_{i}) = \Delta(y_{i}) := w$, for all $i = 1, 2, \cdots,
n$. Now a straightforward computation shows that ${\mathbb
G}_{\mathfrak{h}^{2n+1}} \, (\lambda, \Delta) = \{(\alpha,\,
(\alpha-1) w) ~|~ \alpha \in k^* \}$ and the map $\varphi:
\bigl({\mathbb G}_{\mathfrak{h}^{2n+1}} \, (\lambda, \Delta),
\cdot \bigl)  \to k^{*}$ given by $\varphi(\alpha,\, (\alpha-1) w)
= \alpha$ is a group isomorphism where $\cdot$ is the
multiplication defined by (\ref{grupstrcb}). The conclusion follows by
Corollary~\ref{galcodim1}.
\end{example}

We recall that an extension $\mathfrak{g} \subseteq \mathfrak{h}$
of Lie algebras is called a flag extension \cite[Definition
4.1]{am-2013} if there exists a finite chain of Lie subalgebras of
$\mathfrak{h}$
\begin{equation}\label{lant}
\mathfrak{g} = \mathfrak{h}_0 \subset \mathfrak{h}_1 \subset
\cdots \subset \mathfrak{h}_m = \mathfrak{h}
\end{equation}
such that $\mathfrak{h}_i$ has codimension $1$ in
$\mathfrak{h}_{i+1}$, for all $i = 0, \cdots, m-1$. Supersolvable
Lie algebras provide examples of flag extensions. Based on this
concept, we propose the following definition as the counterpart
for Lie algebras of normal radical extensions of fields:

\begin{definition}\label{radical}
An extension $\mathfrak{g} \subseteq \mathfrak{h}$ of Lie algebras
is called a radical extension if there exists a chain of
subalgebras as in equation~(\ref{lant}) such that each
$\mathfrak{h}_{i-1}$ is invariant with respect to any element
$\tau \in {\rm Gal} (\mathfrak{h}_{i}/\mathfrak{g})$, i.e. $\tau
(\mathfrak{h}_{i-1}) \subseteq \mathfrak{h}_{i-1}$, for all $\tau
\in {\rm Gal} (\mathfrak{h}_i/\mathfrak{g})$ and $i = 1, \cdots,
m$.
\end{definition}

If $\mathfrak{g}$ has codimension $1$ in $\mathfrak{h}$, then
$\mathfrak{h}/\mathfrak{g}$ is a radical extension. Based on
Theorem~\ref{grupulgal} and Corollary~\ref{galcodim1}, exactly as
in the classical case of radical extensions of fields, we can
prove the following:
\begin{theorem}\label{galcodim1aa}
Let $\mathfrak{g} \subseteq \mathfrak{h}$ be a radical extension
of finite dimensional Lie algebras. Then the Galois group ${\rm
Gal} \, (\mathfrak{h}/\mathfrak{g})$ is solvable.
\end{theorem}

\begin{proof} Consider a finite chain of subalgebras as in (\ref{lant}). We will proceed by induction on $m$. If $m = 1$ the conclusion follows
by Corollary~\ref{galcodim1}. Now let $m > 1$ and assume the
statement to be true for $m-1$, that is the group ${\rm Gal} \,
(\mathfrak{h}_{m-1}/\mathfrak{g})$ is solvable. Then, the map:
$$
\Gamma : {\rm Gal} \, (\mathfrak{h}/\mathfrak{g}) \to {\rm Gal} \,
(\mathfrak{h}_{m-1}/\mathfrak{g}), \qquad \Gamma (\tau) :=
\tau_{|_{\mathfrak{h}_{m-1}} }
$$
where $\tau_{|_{\mathfrak{h}_{m-1}} }$ is the restriction of
$\tau$ to $\mathfrak{h}_{m-1}$ is well defined since the extension
is radical, the Lie algebras are finite dimensional and $\Gamma$
is a morphism of groups. Now ${\rm Ker} (\Gamma) = {\rm Gal} (
\mathfrak{h}/\mathfrak{h}_{m-1})$ which is a metabelian (in
particular solvable) group again by Corollary~\ref{galcodim1}.
Thus, we obtain an isomorphism of groups ${\rm Gal} \,
(\mathfrak{h}/\mathfrak{g})/{\rm Gal} \, (
\mathfrak{h}/\mathfrak{h}_{m-1}) \cong {\rm Im} (\Gamma)$, and
${\rm Im} (\Gamma)$ is a solvable group as a subgroup in such a
group. To conclude, we have obtained that ${\rm Gal} \,
(\mathfrak{h}/\mathfrak{g})$ is an extension of a solvable group
${\rm Im} (\Gamma)$ by the solvable group ${\rm Gal} \, (
\mathfrak{h}/\mathfrak{h}_{m-1})$, hence ${\rm Gal} \,
(\mathfrak{h}/\mathfrak{g})$ is solvable too.
\end{proof}

The compatibility condition (\ref{grciud}) which describes the
elements of the group ${\rm Gal} \, (\mathfrak{g}_{(\lambda,
\Delta)}/\mathfrak{g})$ is crucial and deserves a thorough
analysis. First, observe that $(1, \, 0) \in {\mathbb
G}_{\mathfrak{g}} \, (\lambda, \, \Delta)$. On the other hand, if
$(u, \, g_0) \in {\mathbb G}_{\mathfrak{g}} \, (\lambda, \,
\Delta)$, for some $u \neq 1$, then (\ref{grciud}) implies that
$\Delta$ is given by the formula $\Delta (g) = (u-1)^{-1} \,
\bigl( \lambda(g) \, g_0 - [g_0, \, g] \bigl)$, for all $g \in
\mathfrak{g}$. A straightforward computation shows that the second
compatibility condition of (\ref{lambderivari}) is trivially fulfilled, being
equivalent to the Jacobi identity. The center of a Lie
algebra $\mathfrak{g}$ will be denoted by $Z (\mathfrak{g}) := \{
g\in \mathfrak{g} \, | \, [g, \, - ] = 0 \}$. Then $Z
(\mathfrak{g})$ is an abelian subgroup of $(\mathfrak{g}, + )$ and
it can be realized as a Galois group of the following type of Lie
algebra extensions:

\begin{corollary}\label{galperfc}
Let $\mathfrak{g}$ be a Lie algebra and $\Delta \in {\rm Der}
(\mathfrak{g})$ a derivation that is not inner. Then there exists
an isomorphism of groups ${\rm Gal} \, (\mathfrak{g}_{(\Delta)}
/\mathfrak{g}) \cong Z (\mathfrak{g}) $.
\end{corollary}

\begin{proof}
Using (\ref{grciud}) for $\lambda : = 0$, we obtain that $(u,
\, g_0) \in {\mathbb G}_{\mathfrak{g}} \, (\Delta) := {\mathbb
G}_{\mathfrak{g}} \, (0, \, \Delta)$ if and only if $(u-1) \,
\Delta (g) = [g, \, g_0]$, for all $g \in \mathfrak{g}$. Hence,
$(1, \, g_0) \in {\mathbb G}_{\mathfrak{g}} \, (\Delta)$ if and
only if $g_0 \in Z (\mathfrak{g})$. On the other hand, since
$\Delta$ is not inner, it follows that ${\mathbb G}_{\mathfrak{g}}
\, (\Delta)$ does not contain elements of the form $(u, \, g_0)$,
with $u \neq 1$. Now, we apply Corollary~\ref{galcodim1}.
\end{proof}

\begin{example}\label{222}
Let $n \in {\mathbb N}^{*}$ be a positive integer and consider
$\mathfrak{h}^{2n+1}$ to be the $(2n+1)$-dimensional Heisenberg
Lie algebra from Example~\ref{heis}. Then $\Delta :
\mathfrak{h}^{2n+1} \to \mathfrak{h}^{2n+1}$ given by
$\Delta(x_{i}) := y_{i}$, $\Delta(y_{i}) = \Delta(w) := 0$, for
all $i = 1, 2, \cdots, n$ is a derivation of $\mathfrak{h}^{2n+1}$
that is not inner. Furthermore, we denote by $\mathfrak{b}^{2n+2}$
the Lie algebra $\mathfrak{h}^{2n+1}_{(\Delta)}$: it has the
$k$-basis $\{x_{1}, \cdots, x_{n}, \, y_{1}, \cdots, y_{n}, \, w,
\, z\}$ and bracket given for any $i = 1, \cdots, n$ by $[x_{i},
\, y_{i}] = w$, $[z, \, x_{i}] = y_{i}$. By applying
Corollary~\ref{galperfc} and taking into account that $Z
(\mathfrak{h}^{2n+1}) = k w \cong (k, +)$ we obtain that there
exists an isomorphism of groups ${\rm Gal} \, (
\mathfrak{b}^{2n+2}/\mathfrak{h}^{2n+1} ) \cong (k, +)$.
\end{example}

A Lie algebra $\mathfrak{g}$ is called sympathetic if
$\mathfrak{g}$ is perfect, has trivial center and any derivation
is inner. Semisimple Lie algebras over a field of characteristic
zero are sympathetic and there is a sympathetic non-semisimple Lie
algebra in dimension $25$ (see \cite{SZ}).

\begin{corollary}\label{galsym}
Let $\mathfrak{g}$ be a sympathetic Lie subalgebra of codimension
$1$ in a Lie algebra $\mathfrak{h}$. Then there exists an
isomorphism of groups ${\rm Gal} \, (\mathfrak{h}/\mathfrak{g})
\cong k^*$. In particular, if $k$ is a field of characteristic
zero then ${\rm Gal} \, (\mathfrak{gl}(m, k) /\mathfrak{sl} (m,
k)) \cong k^*$.
\end{corollary}

\begin{proof}
Indeed, since $\mathfrak{g}$ is perfect we obtain that
$\mathfrak{h} \cong \mathfrak{g}_{(\Delta)}$, for some derivation
$\Delta$ of $\mathfrak{g}$ \cite[Proposition 2.1]{am-sigma}. Let
$\delta \in \mathfrak{g}$ such that $\Delta = [\delta, \, -]$. By
applying the compatibility condition (\ref{grciud}) for $\lambda : =
0$ and $\Delta := [\delta, \, -]$ we obtain that $(u, \, g_0) \in
{\mathbb G}_{\mathfrak{g}} \, (\Delta)$ if and only if $[g, \, g_0
+ (u-1) \delta] = 0$, for all $g\in \mathfrak{g}$. Since ${\rm
Z}(\mathfrak{g}) = \{0\}$, this is equivalent to the fact that
$g_0 = (1 - u) \delta$. Hence ${\mathbb G}_{\mathfrak{g}} \,
(\Delta)$ consists of all elements of the form $(u, \, (1 - u)
\delta)$, for any $u\in k^*$ and there exists an isomorphism of
groups ${\mathbb G}_{\mathfrak{g}} \, (\Delta) \cong k^*$. Now we
apply Corollary~\ref{galcodim1}.
\end{proof}

All examples of Lie algebra extensions $\mathfrak{h}/\mathfrak{g}$
presented so far have non-trivial Galois group. We end the paper
with a Lie algebra extension whose Galois group is trivial:

\begin{example}\label{galtrivial}
Let $k$ be a field of characteristic $\neq 2$ and consider
$\mathfrak{g}$ to be the perfect $5$-dimensional Lie algebra with
the basis $\{e_{1}, e_{2}, e_{3}, e_{4}, e_{5}\}$ and bracket
given by:
\begin{eqnarray*}
&& [e_{1}, \, e_{2}] = e_{3}, \quad \,\,\, [e_{1}, \, e_{3}] =
-2e_{1},
\quad [e_{1}, \, e_{5}] = [e_{3}, \, e_{4}] = e_{4}\\
&& \left[ e_{2}, \, e_{3} \right] = 2e_{2}, \quad \left[e_{2}, \,
e_{4} \right] = e_{5}, \qquad \,\, \left[e_{3}, \, e_{5} \right] =
- e_{5}
\end{eqnarray*}
It was proven in \cite[Example 3.7]{am-sigma} that the derivation
given in matrix form by $\Delta := e_{11} - e_{41} - e_{22} +
e_{53} - e_{44} - 2\, e_{55}$ is not inner, where $e_{i\,j} \in
\mathcal{M}_{5}(k)$ is the matrix having $1$ in the $(i,j)^{th}$
position and zeros elsewhere. On the other hand, a straightforward
computation shows that $Z (\mathfrak{g}) = \{0\}$. By applying
Corollary~\ref{galperfc} it follows that the extension
$\mathfrak{g} \subseteq \mathfrak{g}_{(\Delta)}$ has trivial
Galois group $\{\rm Id_{\mathfrak{g}_{(\Delta)}}\}$.
\end{example}


\begin{thebibliography}{99}

\bibitem{am-2013}
A.L. Agore, G. Militaru - Extending structures for Lie
algebras, \textit{Monatsh. f\"{u}r Mathematik} {\bf 174} (2014) 169--193.

\bibitem{am-sigma}
A.L. Agore, G. Militaru - Bicrossed products, matched pair
deformations and the factorization index for Lie algebras,
\textit{Symmetry Integr. Geom.} {\bf 10} (2014) 065, 16 pages.

\bibitem{bavula}
V.V. Bavula - The groups of automorphisms of the Witt $W_n$
and Virasoro Lie algebras, \textit{Czech. Math. J.} {\bf 66} 1129--1141.

\bibitem{blattner}
R.J. Blattner, S. Montgomery - Crossed products and Galois
extensions of Hopf algebras, \textit{Pacific J. Math.} {\bf 137} (1989)
37--53.

\bibitem{borel}
A. Borel, G.D. Mostow - On semi-simple automorphisms of Lie
algebras, \textit{Ann. of Math.} {\bf 61}(1955) 389--405.

\bibitem{bra}
A. Brandt - The free Lie ring and Lie representations of the
full linear group, \textit{Trans. AMS} {\bf 56} (1944) 528--536.

\bibitem{bry}
R.M. Bryant, R. Stohr - On the module structure of free Lie
algebras, \textit{Trans. AMS} {\bf 352} (2000) 901--934.

\bibitem{CE}
C. Chevalley, S. Eilenberg - Cohomology theory of Lie groups
and Lie algebras, \textit{Trans. AMS} {\bf 63} (1948) 85--124.

\bibitem{Chr}
T. Christodoulakis, G.O. Papadopoulos, A. Dimakis -
Automorphisms of real four-dimensional Lie algebras and the
invariant characterization of homogeneous 4-spaces, \textit{J. Phys.
A-Math. Gen.} {\bf 36} (2003), 427--441.

\bibitem{donkin}
S. Donkin, K. Erdmann - Tilting modules, symmetric functions
and the module structure of the free Lie algebra, \textit{J. Algebra} {\bf
203} (1998) 69--90.

\bibitem{Meyer}
F. DeMeyer, E. Ingraham - Separable algebras over commutative
rings, Springer Verlag, Belin (1971).

\bibitem{Eld}
A. Elduque, M. Kochetov - Gradings on simple Lie algebras
Mathematical Surveys and Monographs AMS {\bf 189} (2013).

\bibitem{fisher}
D.J. Fisher, R.J. Gray, P.E. Hydon - Automorphisms of real Lie
algebras of dimension five or less, \textit{Journal of Physics A:
Mathematical and General} {\bf 46} (2013) 225204.

\bibitem{gray}
R.J. Gray - Automorphisms of Lie algebras, PhD thesis, Univ.
of Surrey (2013). Available online at:
https://www.surrey.ac.uk/maths/files/pgr/Theses/thesis$_{-}$RGray.pdf

\bibitem{Han}
J. Han, Y. Su - Automorphism groups of Witt algebras,
arXiv:1502.01441.

\bibitem{hochster}
M. Hochster - Invariant theory of commutative rings, in:
Group actions on rings, S. Montgomery (ed.), \textit{Contemporary Math.
AMS} {\bf 43} (1985), 161--179.

\bibitem{H}
J.E. Humphreys - Introducution to Lie algebras and
Representation theory, Spinger (1972).

\bibitem{jacobson}
N. Jacobson - Lie algebras, Dover Publications NY (1962).

\bibitem{Kac}
V. Kac - Graded Lie algebras and symmetric spaces (Russian),
\textit{Functional. Anal. i Prilozen} {\bf 2} (1968), 93--94.

\bibitem{kni}
V. Knibbeler - Invariants of Automorphic Lie Algebras,
arXiv:1504.03616.

\bibitem{Koch}
M. Kochetov - Gradings on Finite-Dimensional Simple Lie
Algebras, \textit{Acta Applicandae Mathematicae} {\bf 108} (2009),
101--127.

\bibitem{procesi}
H. Kraft, C. Procesi - Classical Invariant Theory, a Primer,
Available online at:
http://jones.math.unibas.ch/~kraft/Papers/KP-Primer.pdf

\bibitem{Lang}
S. Lang - Algebra, 3rd editon, Graduate Texts in Mathematics,
Springer (2005).

\bibitem{lomb}
S. Lombardo, A.V. Mikhailov - Reduction Groups and Automorphic
Lie Algebras, \textit{Comm. Math. Phys.} {\bf 258} (2005), 179--202.

\bibitem{LW}
J.H. Lu, A. Weinstein - Poisson Lie groups, dressing
transformations and Bruhat decompositions, \textit{J. Differential Geom.}
{\bf 31}(1990), 501--526.

\bibitem{Amajid}
A. Majid - Lectures on Differential Galois Theory, University
Lecture Series AMS {\bf 7}(1997).

\bibitem{majid}
S. Majid - Physics for algebraists: non-commutative and
non-cocommutative  Hopf algebras by a bicrossproduct
construction, \textit{J. Algebra} {\bf 130} (1990), 17--64.

\bibitem{mat}
A. Mathew - The Galois group of a stable homotopy theory,
\textit{Adv. Math.} {\bf 291} (2016), 403--541.

\bibitem{mont}
S. Montgomery - Hopf Algebras and Their Actions on Rings, AMS
(1993).

\bibitem{mont2}
S. Montgomery - Fixed Rings of Finite Automorphism Groups of
Associative Rings, Springer Lecture Notes in Mathematics {\bf
818} (1980).

\bibitem{ocneanu}
A. Ocneanu - Quantized groups, string algebras and Galois
theory for algebras, in: Operator Algebras and Applications,
London Math. Soc. Lecture Notes, {\bf 136} (1989), 119--172.

\bibitem{papistas}
A.I. Papistas - On the module structure of a group action on a
Lie algebra, \textit{J. Aust. Math. Soc.} {\bf 82} (2007), 237--248.

\bibitem{petros}
T. Christodoulakis, P.A. Terzis - The general solution of
Bianchi type III vacuum cosmology, \textit{J. Physics: Conference Series}
{\bf 68} (2007), 012039.

\bibitem{rognes}
J. Rognes - Galois Extensions of Structured Ring
Spectra/Stably Dualizable Groups, \textit{Memoirs AMS} {\bf 192} (2008),
137 pp.

\bibitem{SZ}
Y. Su, L. Zhub - Derivation algebras of centerless perfect Lie
algebras are complete, \textit{J. Algebra} {\bf 285} (2005), 508--515.

\bibitem{svob}
M. Svobodova - Fine Gradings of Low-Rank Complex Lie Algebras
and of Their Real, \textit{Symmetry Integr. Geom.} {\bf 4}(2008), 039, 13
pages.

\bibitem{thrall}
R.M. Thrall - On symmetrized Kronecker powers and the
structure of the free Lie ring, \textit{Amer. J. Math.} {\bf 64} (1942),
371--388.

\bibitem{Wever}
F. Wever - \"{U}ber Invarianten von Lieschen Ringen, \textit{Math.
Annalen} {\bf 120} (1949), 563--580.

\end{thebibliography}
\end{document}